\documentclass[10pt, oneside]{article}
\usepackage{mathrsfs}
\usepackage{amsfonts}
\usepackage{amsmath,amsfonts,amssymb,amsthm}
\usepackage{caption}
\usepackage{subfigure}
\usepackage{cite}
\usepackage{color}
\usepackage{graphicx}
\usepackage{subfigure}
\usepackage{float}
\usepackage{paralist}
\usepackage{indentfirst}
\usepackage{authblk}

\def\var{\varepsilon}

\def\beq{\begin{equation}}
	\def\eeq{\end{equation}}
\def\bsp{\begin{aligned}}
	\def\esp{\end{aligned}}

\usepackage[colorlinks,
linkcolor=red,
citecolor=blue
]{hyperref}

\usepackage[T1]{fontenc}
\DeclareMathOperator{\dive}{div}

\newtheorem{theorem}{Theorem}[section]
\newtheorem{lemma}{Lemma}[section]
\newtheorem{prop}{Proposition}[section]

\newtheorem{remark}{Remark}[section]

\newtheorem{defn}{Definition}[section]

\def\bma#1\ema{{\allowdisplaybreaks\begin{aligned}#1\end{aligned}}}

\numberwithin{equation}{section}

\begin{document}
	\title{{\LARGE \textbf{Global existence and large time behavior of weak solutions to the two-phase flow}}}
	
	
	\author[a,b]{Ya-Ting Wang \thanks{E-mail: wangyating1009@sina.cn(Y.-T Wang).}}
	
	\author[a,b]{Ling-Yun Shou \thanks{E-mail: shoulingyun11@gmail.com(L.-Y Shou).}}
	\affil[a]{School of Mathematical Sciences,
		Capital Normal University, Beijing 100048, P.R. China}
	\affil[b]{Academy for Multidisciplinary Studies, Capital Normal University, Beijing 100048, P.R. China}
	
	\date{}
	\renewcommand*{\Affilfont}{\small\it}
	\maketitle
	\begin{abstract}
		In this paper, we consider a two-phase flow model consisting of the compressible Navier-Stokes equations with degenerate viscosity coupled with the compressible Navier-Stokes equations with constant viscosities via a drag force, which can be derived from Chapman-Enskog expansion for the compressible Navier-Stokes-Vlasov-Fokker-Planck system. For general initial data, we establish the global existence of weak solutions with finite energy to the initial value problem in the three-dimensional periodic domain, and prove the convergence of global weak solutions to its equilibrium state as the time tends to infinity.
	\end{abstract}
	\section{Introduction}

	The two-phase flow models can simulate a variety of physical phenomena describing the mixture of two different flows with appropriate interactions, and play an important role in many applied scientific areas, such as nuclear, chemical-process, petroleum, cryogenic, bio-medical, oil-and-gas, microtechnology, and so on \cite{brennen1,desv1, gidaspow1,ishii2,zuber1}.  In the present paper, we consider the initial value problem (IVP) for the following two-phase flow model in the periodic domain $\mathbb{T}^3:=\mathbb{R}^{3}/\mathbb{Z}^3$:
	\begin{equation}\label{two}
		\left\{
		\begin{aligned}
			&n_{t}+\dive(n v)=0,\\
			&(nv)_{t}+\dive(n v\otimes v)+\nabla n=-\kappa n(v-u)+\eta\dive( n\mathbb{D}(v)),\\
			&\rho_{t}+\dive(\rho u)=0,\\
			&(\rho u)_{t}+\dive (\rho  u\otimes u)+\nabla P(\rho)=\kappa n(v-u)+\mu\Delta u+(\mu+\lambda)\nabla\dive u,\quad x\in \mathbb{T}^3,~~t>0,
		\end{aligned}
		\right.
	\end{equation}
	with the initial data
	\begin{equation}
		\begin{aligned}
			(n, n v,\rho,\rho u)(x,0)=(n_{0},m_{0},\rho_{0},\tilde{m}_{0})(x),\quad x\in \mathbb{T}^3,\label{d}
		\end{aligned}
	\end{equation}
	where $n=n(x,t)\geq0$ and $v=v(x,t)\in\mathbb{R}^{3}$ denote the density and velocity of compressible Navier-Stokes equations  $(\ref{two})_{1}$-$(\ref{two})_{2}$ with degenerate viscosity, and $\rho=\rho(x,t)\geq0$ and $u=u(x,t)\in \mathbb{R}^3$ stand for the density and velocity of compressible Navier-Stokes equations  $(\ref{two})_{3}$-$(\ref{two})_{4}$ with constant viscosities, $\kappa n(v-u)$ is the drag force term. $\mathbb{D}(v):=\frac{\nabla v+(\nabla v)^{tr}}{2}$ is the deformation tensor. The coefficients $\kappa$, $\eta$, $\mu$ and $\lambda$ are constants satisfying
	\begin{equation}\nonumber
		\begin{aligned}
			\kappa>0,\quad\quad \eta>0,\quad \quad \mu>0,\quad\quad 2\mu+\lambda>0.
		\end{aligned}
	\end{equation}
	The pressure $P(\rho)$ takes the form 
	\begin{equation}\nonumber
		\begin{aligned}
			&P(\rho)=A\rho^{\gamma},
		\end{aligned}
	\end{equation}
	with $\gamma>\frac{3}{2}$ the adiabatic exponent and $A>0$ a constant.

	The two-phase flow model (\ref{two}) can be derived from the compressible Navier-Stokes-Vlasov-Fokker-Planck equations with a local
	alignment force as follows
	\begin{equation}\label{NSVFP}
		\left\{
		\begin{aligned}
			&f_{t}+\xi\cdot \nabla_{\xi} f=\dive_{\xi}((\xi-u)f)+\frac{1}{\eta} \dive_{\xi} ((\xi-v)+\nabla_{\xi}f),\\
			&\rho_{t}+\dive_{x}(\rho u)=0,\\
			&(\rho u)_{t}+\dive_{x} (\rho  u\otimes u)+\nabla_{x} P(\rho)\\
			&\quad=\kappa n(v-u)+\mu\Delta_{x} u+(\mu+\lambda)\nabla_{x}\dive_{x} u,\quad (x,\xi)\in \mathbb{T}^3\times\mathbb{R}^3,\quad t>0,
		\end{aligned}
		\right.
	\end{equation}
	where $f$ is the distribution function associated with the particles, and $n$ and $nv$ are the macroscopical density and momentum defined by
	$$
	n(x,t):=\int_{\mathbb{R}^3} f(x,\xi,t)d\xi,\quad\quad nv(x,t):=\int_{\mathbb{R}^3}\xi f(x,\xi,t)d\xi.
	$$
	The fluid-particle system (\ref{NSVFP}) arises in modelling the sedimentation of suspensions, sprays, combustion \cite{jabin1,o1,will1,will2}, etc. Recently, Li-Wang-Wang \cite{lhl1} applied the Chapman-Enskog expansion for the fluid-particle model (\ref{NSVFP}) around the local Maxwellian 
	$$
	\frac{n(x,t)}{\sqrt{(2\pi)^3}}e^{-\frac{|\xi-v(x,t)|^2}{2}}
	$$
	to obtain the two-phase flow equations
	\begin{equation}\label{two1}
		\left\{
		\begin{aligned}
			&n_{t}+\dive_{x}(n v)=0,\\
			&(nv)_{t}+\dive_{x}(n v\otimes v)+\nabla_{x} n=-\kappa n(v-u)+\eta \dive_{x}( n\mathbb{D}_{x}(v))-\eta\int_{\mathbb{R}^3}\xi\otimes \xi \cdot \nabla_{x} \Pi d\xi,\\
			&\rho_{t}+\dive(\rho u)=0,\\
			&(\rho u)_{t}+\dive_{x} (\rho  u\otimes u)+\nabla_{x} P(\rho)=\kappa n(v-u)+\mu\Delta_{x} u+(\mu+\lambda)\nabla_{x}\dive_{x} u,
		\end{aligned}
		\right.
	\end{equation}
	where the term $\Pi$ is governed by the microscopic part. 
		
		In the limiting case $\eta\rightarrow 0$, the system (\ref{two1}) converges to the coupled Euler-Navier-Stokes two-phase flow model
	\begin{equation}\label{EulerNS}
		\left\{
		\begin{aligned}
			&n_{t}+\dive(n v)=0,\\
			&(nv)_{t}+\dive(n v\otimes v)+\nabla n=-\kappa n(v-u),\\
			&\rho_{t}+\dive(\rho u)=0,\\
			&(\rho u)_{t}+\dive (\rho  u\otimes u)+\nabla P(\rho)=\kappa n(v-u)+\mu\Delta u+(\mu+\lambda)\nabla\dive u,\quad x\in \mathbb{T}^3,~t>0.
		\end{aligned}
		\right.
	\end{equation}	
	
	When the microscopic effect in (\ref{two1}) is considered to be suitably small, we can derive the two-phase flow model (\ref{two}) as a suitable approximate system of both (\ref{NSVFP}) and (\ref{EulerNS}). To our best
	knowledge, there are few results available on mathematical analysis for the two-phase flow model (\ref{two}). In \cite{lhl2, lhl3}, the authors showed the existence and nonlinear stability of stationary solutions to the inflow/outflow problem in a half line.


	The coupled Euler-Navier-Stokes two-phase flow model $(\ref{EulerNS})$ has been derived rigorously in \cite{choi3} as a fluid-dynamical limit of the compressible Navier-Stokes-Vlasov-Fokker-Planck equations (\ref{NSVFP}). Global well-posedness and time-decay rates of strong solutions to (\ref{EulerNS}) for initial data close to a constant equilibrium state have been investigated both in Sobolev spaces \cite{choi1,jung1,wu1} and in critical Besov spaces \cite{lhl0}. In addition, there are many important works on other related two-phase flow models; refer to \cite{bresch4,bresch5,bresch2,wen0,novotny2} and references therein.

The two-phase flow system \eqref{two} for $\kappa=0$ can reduce to the compressible Navier-Stokes equations 
	\begin{equation}\label{NS}
		\left\{
		\begin{aligned}
			&\rho_{t}+\dive(\rho u)=0,\\
			&(\rho u)_{t}+\dive (\rho  u\otimes u)+\nabla P(\rho)=\dive (\mu(\rho)\mathbb{D}u+\lambda(\rho)\dive u \mathbb{I} ).
		\end{aligned}
		\right.
	\end{equation}	
	We would like to mention some important progress made recently on the existence of global weak solutions with large aptitude for the compressible Navier-Stokes equations \eqref{NS}. In the case that both $\mu(\rho)$ and $\lambda(\rho)$ are constants, Lions first established the global existence of weak solutions to the $d$-dimensional ($d\geq2)$ compressible Navier-Stokes equations (\ref{NS}) for general initial data, where the pressure satisfies the $\gamma$-law $P(\rho)=\rho^{\gamma}$ with $\gamma\geq \frac{3d}{d+2}$ $(d=2,3)$ and $\gamma>\frac{d}{2}$ $(d\geq 4)$. The range for the adiabatic constant has been relaxed to any $\gamma>\frac{d}{2}$ by Feireisl-Novotn$\rm{\acute{y}}$-Petzeltov$\rm{\acute{a}}$ \cite{feireisl1}, any $\gamma>1$  by Jiang-Zhang \cite{jiang1} if the initial data is spherically symmetric, and $\gamma=1$ by Plotnikov-Weigant \cite{plotnikov1} in the two-dimensional case. Indeed, Hu \cite{hu1} studied the concentration phenomenon of the convective term $\rho u\otimes u$ for $\gamma\in [1,\frac{d}{2}]$. In addition, Feireisl \cite{feireisl2} showed that such weak solutions asymptotically converge to its equilibrium state as the time grows up, and then Feireisl \cite{feireisl2} proved the global existence of weak solutions to \eqref{NS} for general pressure laws allowed to be non-monotone on a compact set. Bresch-Jabin \cite{bresch1} developed new compactness estimates of the density and obtained global weak solutions to \eqref{NS} for thermodynamically unstable pressure laws and anisotropic viscosity coefficients.

	When $\mu(\rho)$ and $\lambda(\rho)$ depend on the density which are degenerate at vacuum, such viscous compressible equations appear in the description of shallow water \cite{ge1,mar1} or geophysical flows \cite{pe1,lions2}, and can be derived from the fluid-dynamical approximation to the Boltzmann equation \cite{liu1}. Bresch et. al. \cite{bresch30,bresch10,bresch20} showed that for $\mu(\rho)$ and $\lambda(\rho)$ verifying $\lambda(\rho)=\rho\mu'(\rho)-\mu(\rho)$, the system (\ref{NS}) had a new entropy inequality providing higher regularity of the density, which was applied to prove the global existence of weak solutions for compressible Navier-Stokes equations with force terms \cite{bresch10,bresch20}. Mellet-Vasseur \cite{mellet1} investigated the weak stability of global solutions to \eqref{NS} by deriving a $L^{\infty}(0,T;L\log{L})$-type estimate of the velocity.  Li-Xin \cite{l1} and Vasseur-Yu \cite{vasseur10} used two different ways to construct the approximate sequence and established the global existence of weak solutions to (\ref{NS}) independently. Bresch-Vasseur-Yu \cite{bresch40} obtained global weak solutions for more general viscous stress tensors.


	In this paper, we aim to establish the existence and large time behavior of global weak solutions with finite energy to the IVP (\ref{two})-(\ref{d}) for general initial data.

	First, we give the definition of global weak solutions to the IVP $(\ref{two})$-$(\ref{d})$ below.
	\begin{defn}\label{defn11}
		$(n,nv,\rho,\rho u)$ with $n\geq 0$ and $\rho\geq0$ is said to be a global weak solution to the IVP $(\ref{two})$-$(\ref{d})$ if for any time $T>0$, the following properties hold$:$
		\begin{itemize}
			\item{} Integrability conditions.
			
			\begin{equation}\nonumber
				\left\{
				\begin{aligned}
					& n\in L^{\infty}(0,T;L^{1}(\mathbb{T}^3)),\quad\quad \quad ~\nabla \sqrt{n}\in L^{\infty}(0,T;L^2(\mathbb{T}^3)),\\
					&\rho \in L^{\infty}(0,T;L^{\gamma}(\mathbb{T}^{3} )),\quad\quad \quad~ \rho^{\frac{5}{3}\gamma-1}\in L^{1}(0,T;L^{1}(\mathbb{T}^3)),\\
					&\sqrt{n} v\in L^{\infty}(0,T;L^2(\mathbb{T}^3)),\quad \quad\sqrt{\rho} u\in L^{\infty}(0,T, L^2(\mathbb{T}^{3})),\\
					& u\in L^2(0,T;H^1(\mathbb{T}^{3}) ), \quad\quad \quad~\sqrt{n}v-\sqrt{n}u\in L^2(0,T;L^2(\mathbb{T}^3)).\\
				\end{aligned}
				\right.
			\end{equation}
			
			\item Equations.
			
			For any text function $\phi\in \mathcal{D}(\mathbb{T}^3\times [0,T))$, the equations $(\ref{two})_{1}$, $(\ref{two})_{2}$ and $(\ref{two})_{4}$ are satisfied in the following sense$:$
			\begin{align}
				&\int_{\mathbb{T}^3}n_{0}\phi(0)dx+\int_{0}^{T}\int_{\mathbb{T}^3} (n \phi_{t}+n v \cdot \nabla \phi)dxdt=0,\label{e1}\\
				&\int_{\mathbb{T}^3} m_{0}^{i}\phi(0)dx+\int_{0}^{T}\int_{\mathbb{T}^3} [  n v^{i} \phi_{t}+\sqrt{n} v^{i} \sqrt{n} v\cdot \nabla \phi-\partial_{i}n \phi]dx\nonumber\\
				&\quad=\int_{0}^{T}\int_{\mathbb{T}^3} [\kappa n v^i\phi-\kappa n u^i\phi-\eta n v^i \Delta \phi-\eta nv^i \nabla \partial_{i}\phi]dxdt,\quad\quad\quad\quad~\quad i=1,2,3,\label{e2}\\
				&\int_{\mathbb{T}^3} \tilde{m}_{0}^{i}\phi(0)dx+\int_{0}^{T}\int_{\mathbb{T}^3} [ \rho u^{i} \phi_{t}+\rho  u^i u\cdot \nabla \phi+A\rho^{\gamma}\partial_{i}\phi]dx\nonumber\\
				&\quad=\int_{0}^{T}\int_{\mathbb{T}^3} [\kappa nu^i\phi-\kappa nv^i\phi+\mu \nabla u^i \nabla \phi+(\mu+\lambda)\dive u \partial_{i}\phi]dxdt,\quad~~ \quad i=1,2,3,\label{e4}
			\end{align}
			and the equation $(\ref{two})_{3}$ holds in in the sense of renormalized solutions, i.e., for $b\in C^{1}(\mathbb{R})$ satisfies $b'(z)=0$ with $z\in\mathbb{R}$ large enough,
			\begin{align}
				&\int_{\mathbb{T}^3}b(\rho_{0})\phi(0)dx+\int_{0}^{T}\int_{\mathbb{T}^3} [b(\rho) \phi_{t}+b(\rho)u \cdot \nabla \phi-(\rho b'(\rho)-b(\rho))\dive u\phi]dxdt=0.\label{e3}
			\end{align}
			In addition, it holds for any text function $\varphi\in \mathcal{D}(\mathbb{T}^3)$ that
			\begin{equation}\label{initial}
					\left\{
				\begin{aligned}
					&\lim_{t\rightarrow0}\int_{\mathbb{T}^3} n\varphi dx=\int_{\mathbb{T}^3}n_{0}\varphi dx,\quad\quad \lim_{t\rightarrow0}\int_{\mathbb{T}^3} nv\varphi dx=\int_{\mathbb{T}^3}m_{0}\varphi dx,\\
					&\lim_{t\rightarrow0}\int_{\mathbb{T}^3} \rho \varphi dx=\int_{\mathbb{T}^3}\rho_{0}\varphi dx,\quad\quad \lim_{t\rightarrow0}\int_{\mathbb{T}^3} \rho u\varphi dx=\int_{\mathbb{T}^3}\tilde{m}_{0}\varphi dx.
				\end{aligned}
				\right.
			\end{equation}
			
			\item Energy inequality.
                                  
For a.e. $t\in(0,T)$, there holds
			\begin{equation}\label{energyin}
				\begin{aligned}
					&\int_{\mathbb{T}^3}\big(\frac{1}{2}n|v|^2+n\log{n}-n+1+\frac{1}{2}\rho |u|^2+\frac{A\rho^{\gamma}}{\gamma-1}\big)dx\\
					&\quad\quad+\int_{0}^{t}\int_{\mathbb{T}^3} \big( \kappa n|v-u|^2+\mu |\nabla u|^2+(\mu+\lambda)(\dive u)^2 \big)dxd\tau\\
					&\quad\leq \int_{\mathbb{T}^3}(\frac{1}{2}\frac{|m_{0}|^2}{n_{0}}+n_{0}\log{n_{0}}-n_{0}+1+\frac{1}{2}\frac{|\tilde{m}_{0}|^2}{\rho_{0}}+\frac{A\rho_{0}^{\gamma}}{\gamma-1})dx.
				\end{aligned}
			\end{equation}
		\end{itemize}
	\end{defn}

	\vspace{1ex}
	
	Then, we have the global existence of weak solutions to the IVP (\ref{two})-(\ref{d}) as follows.
	\begin{theorem}\label{theorem11}
		Assume that the initial data $(n_{0},m_{0},\rho_{0},\tilde{m}_{0})$ satisfies
		\begin{equation}\label{a1}
			\left\{
			\begin{aligned}
				&0\leq n_{0}\in L^1(\mathbb{T}^3),\quad n_{0}\log{n_{0}}\in L^1(\mathbb{T}^3),\quad \nabla\sqrt{ n_{0}}\in L^2(\mathbb{T}^3), \\
				&m_{0}=0~ \text{in}~\{x\in\mathbb{T}^3~|~ n_{0}(x)=0\},\quad \frac{|m_{0}|^2}{n_{0}}\in L^1(\mathbb{T}^3),  \quad \frac{m_{0}^{2+\eta_{0}}}{n_{0}^{1+\eta_{0}}}\in L^{1}(\mathbb{T}^3),\\
				&  0\leq \rho_{0}\in L^{\gamma}(\mathbb{T}^3)\cap L^{\gamma}(\mathbb{T}^3),\\
				&\tilde{m}_{0}=0~ \text{in}~\{x\in\mathbb{T}^3~| ~\rho_{0}(x)=0\},\quad  \frac{|\tilde{m_{0}}|^2}{\rho_{0}} \in L^1(\mathbb{T}^3),
			\end{aligned}
			\right.
		\end{equation}
		with $\eta_{0}>0$ any small constant. Then the IVP $(\ref{two})$-$(\ref{d})$ admits a global weak solution $(n,nv, \rho, \rho u)$ in the sense of Definition \ref{defn11}. Moreover, the solution $(n,nv, \rho, \rho u)$ satisfies for any time $T>0$ that
		\begin{equation}\label{mass}
		\left\{
		\begin{aligned}
			&\int_{\mathbb{T}^3} ndx=\int_{\mathbb{T}^3}n_{0}dx,\quad\quad~ \int_{\mathbb{T}^3} \rho dx=\int_{\mathbb{T}^3} \rho_{0}dx,\quad t\in(0,T),\\
			&\int_{\mathbb{T}^3} (n v+\rho u )dx=\int_{\mathbb{T}^3}(m_{0}+\tilde{m}_{0})dx,\quad \quad \quad \quad ~\quad t\in(0,T),
		\end{aligned}
		\right.
		\end{equation}
		and
		\begin{equation}\label{BD}
		\left\{
		\begin{aligned}
			&\underset{t\in[0, T]}{{\rm{ess~sup}}}~\int_{\mathbb{T}^3} |\nabla \sqrt{n}|^2dx +\int_{0}^{T}\int_{\mathbb{T}^3} |\nabla \sqrt{n}|^2dxdt \le C, \\
				&\underset{t\in[0, T]}{{\rm{ess~sup}}}~\int_{\mathbb{T}^3} n(1+|v|^2)\log{(1+|v|^2)}dx\leq C_{T}, 
		\end{aligned}
		\right.
		\end{equation}
		where $C>0$ is a constant independent of the time $T>0$, and $C_{T}>0$ is a constant dependent of the time $T>0$.

	\end{theorem}
	
	
	
	\begin{remark}
		By similar arguments, Theorem \ref{theorem11} can be extended to the two-dimensional case for any adiabatic constant $\gamma>1$.
	\end{remark}
	
	Next, we study the large time behavior of global weak solutions to the IVP (\ref{two})-(\ref{d}).
	
	\begin{theorem}\label{theorem12}
		Let the assumptions $(\ref{a1})$ be satisfied, and $(n,nv,\rho,\rho u)$ be the global weak solution to the IVP $(\ref{two})$-$(\ref{d})$ given by Theorem \ref{theorem11}. Then $(n,v,\rho,u)$ converges to its equilibrium state $(n_{c},u_{c},\rho_{c},u_{c})$ in the sense
		\begin{equation}\label{decay1}
			\begin{aligned}
				&\lim_{t\rightarrow \infty}\int_{\mathbb{T}^3} \big{(} n|v-u_{c}|^2+|n-n_{c}|^{p}+\rho|u-u_{c}|^2+|\rho-\rho_{c}|^{\gamma}\big{)}dx=0,\quad p\in [1,3),
			\end{aligned}
		\end{equation}
		where the constants $n_{c}$, $\rho_{c}$ and $u_{c}$ are denoted by
		\begin{equation}\label{equ}
			\begin{aligned}
				&n_{c}:=\int_{\mathbb{T}^3} n_{0}dx,\quad\quad \rho_{c}:=\int_{\mathbb{T}^3} \rho_{0}dx,\quad\quad u_{c}:=\frac{\int_{\mathbb{T}^3}(m_{0}+\tilde{m}_{0})dx}{\int_{\mathbb{T}^3} (n_{0}+\rho_{0})dx}.
			\end{aligned}
		\end{equation}
	\end{theorem}

	\begin{remark}
		Theorem \ref{theorem12} implies that the velocities of the compressible Navier-Stokes equations $\eqref{two}_{1}$-$\eqref{two}_{2}$ with degenerate viscosity and the compressible Navier-Stokes equations $\eqref{two}_{3}$-$\eqref{two}_{4}$ with constant viscosities are aligned when the time tends to infinity.
	\end{remark}


	\vspace{2ex}

	The rest part of this paper is arranged as follows. In Section \ref{section2}, we state the global well-posedness of strong solutions to the approximate system. Section \ref{section3} is devoted to the a-priori estimates of approximate solutions. In Section \ref{section4}, we vanish the artificial viscosities and show the convergence of approximate sequence to a global weak solution for (\ref{two}) with an artificial pressure term. In Section \ref{section5}, we vanish the artificial pressure and prove Theorem \ref{theorem11} concerning global existence of weak solutions to the original IVP (\ref{two})-(\ref{d}). Theorem \ref{theorem12} on the large time behavior of global weak solutions will be shown in Section \ref{section6}.

	\section{Approximate sequence}\label{section2}
	
Inspired by Li-Xin \cite{l1}, for $\var$, $\delta\in(0,1)$ and $\gamma_{0}>\max\{\gamma+4\}$, we are ready to solve the following approximate problem:
	\begin{equation}\label{twoapp}
		\left\{
		\begin{aligned}
			&n_{t}+\dive (nv)=\var \sqrt{n}\Delta \sqrt{n}+\var\sqrt{n}\dive (|\nabla \sqrt{n}|^2\nabla \sqrt{n}) +\var n^{-12},\\
			&n(v_{t}+v\cdot\nabla v)+\nabla n+\varepsilon n|v|^{3}v+\var n^{-12}v\\
			&\quad\quad~= -\kappa n(v-u)+\eta \dive (n\mathbb{D} (v))+\sqrt{\varepsilon} \dive (n\nabla v)+\varepsilon\sqrt{n} |\nabla \sqrt{n}|^2\nabla \sqrt{n}\cdot \nabla v,\\
			&\rho_{t}+\dive(\rho u)=\varepsilon \Delta \rho,\\
			&\rho(u_{t} +  u\cdot\nabla u)+\nabla(A\rho^{\gamma}+\delta \rho^{\gamma_{0}})+\var   |u|^{8}u\\
			&\quad\quad~=\kappa n(v-u)+\mu\Delta u+(\mu+\lambda)\nabla\dive u+\varepsilon\nabla u\cdot\nabla \rho,\quad x\in\mathbb{T}^{3},\quad t>0,\\
		&(n,v,\rho, u)(x,0)=(n_{0,\delta},v_{0,\delta},\rho_{0,\delta},u_{0,\delta})(x),\quad x\in\mathbb{T}^{3}.
		\end{aligned}
		\right.
	\end{equation}
 The regularized initial data $(n_{0,\delta},v_{0,\delta},\rho_{0,\delta},u_{0,\delta})$ is constructed by
 	\begin{equation}\label{dapp}
	\left\{
		\begin{aligned}
			&n_{0,\delta}:= (\sqrt{n_{0}}\ast j_{\delta})^2+\delta^{\frac{1}{100}},~~ v_{0,\delta}:=\frac{ ( n_{0}^{-\frac{1+\eta_{0}}{2+\eta_{0}}} m_{0})\ast j_{\delta}}{n_{0,\delta}^{\frac{1}{2+\eta_{0}}}},\\
			&\rho_{0,\delta}:= \rho_{0}\ast j_{\delta}+\delta\geq \delta,\quad u_{0,\delta}:=  \frac{\frac{\tilde{m}_{0}}{\sqrt{\rho_{0}}}\ast j_{\delta}}{\sqrt{\rho_{0,\delta}}},
		\end{aligned}
		\right.
	\end{equation}
 where $j_{\delta}$ is the smooth function satisfying
	\begin{equation}\label{jdelta}
		\left\{
		\begin{aligned}
			& \|j_{\delta}\|_{L^1 }=1,\quad 0\leq j_{\delta} \leq \delta ^{-\frac{1}{2\gamma_{0}}},\quad |\nabla j_{\delta}|\leq C\delta^{-\frac{1}{8}} j_{\delta},\\
			&j_{\delta}\ast f\rightarrow f \quad~~\text{in}~ L^{p}(\mathbb{T}^{3}),\quad\text{as}\quad \delta\rightarrow 0,\quad\forall f\in L^{p}(\mathbb{T}^{3}) ,\quad p\in [1,\infty),
		\end{aligned}
		\right.
	\end{equation}
	with some constant $C>0$ independent of $\delta$. 
	
	It is easy to verify as $\delta\rightarrow 0$ that
	\begin{equation}\label{dapp21}
		\left\{
		\begin{aligned}
		       &\rho_{0,\delta}\rightarrow \rho_{0},\quad\quad\quad\quad\quad\quad\quad\quad\quad\quad\quad\quad\quad   \quad\text{in}\quad L^{\gamma}(\mathbb{T}^3),\\
			&\rho_{0,\delta} |u_{0,\delta}|^2\rightarrow \frac{|m_{0}|^2}{\rho_{0}}\quad\quad\quad\quad\quad\quad\quad \quad\quad \quad~ \text{in}\quad L^1(\mathbb{T}^3),\\
			&n_{0,\delta}\rightarrow n_{0}\quad\quad \quad\quad\quad\quad\quad\quad\quad\quad\quad\quad\quad ~\quad\text{in}\quad L^1(\mathbb{T}^3),\\
			&\nabla \sqrt{n_{0,\delta}}\rightarrow \nabla \sqrt{n_{0}}\quad \quad\quad\quad\quad\quad\quad\quad\quad\quad\quad\text{in}\quad L^2(\mathbb{T}^3),\\
			&n_{0,\delta} |v_{0,\delta}|^{2+\eta_{0}}\rightarrow \frac{|m_{0}|^{2+\eta_{0}}}{n_{0}^{1+\eta_{0}}}\quad \quad\quad\quad\quad\quad\quad~\text{in}\quad L^1(\mathbb{T}^3),\\
			&n_{0,\delta} |v_{0,\delta}|^2=n_{0,\delta}^{\frac{\eta_{0}}{2+\eta_{0}}}|( n_{0}^{-\frac{1+\eta_{0}}{2+\eta_{0}}} m_{0})\ast j_{\delta}|^2\\
			&\quad\quad\quad\quad~\rightarrow n_{0}^{\frac{\eta_{0}}{2+\eta_{0}}}|n_{0}^{-\frac{1+\eta_{0}}{2+\eta_{0}}} m_{0}|^2=\frac{|m_{0}|^2}{n_{0}}\quad \text{in}\quad L^1(\mathbb{T}^3).
		\end{aligned}
		\right.
	\end{equation}
	Denote 
	\begin{equation}
			\begin{aligned}
				&E_{0,\delta}:=\int_{\mathbb{T}^3}\big(\frac{1}{2}n_{0,\delta}|v_{0,\delta}|^2+n_{0,\delta}\log{n_{0,\delta}}-n_{0,\delta}+1+\frac{1}{2}\rho_{0,\delta} |u_{0,\delta}|^2+\frac{A\rho_{0,\delta}^{\gamma}}{\gamma-1}\\
				&\quad\quad\quad+\frac{\delta \rho_{0,\delta}^{\gamma_{0}}}{p_{0}-1}+\var_{0,\delta} n^{-12}\big)dx.\label{E0delta}
			\end{aligned}
		\end{equation}
	By (\ref{jdelta}) and (\ref{dapp21}), it follows that
	\begin{equation}\label{dapp22}
		\begin{aligned}
		        &E_{0,\delta}\leq C,\quad\quad\lim_{\delta\rightarrow0} E_{0,\delta}=0,\quad\quad \int_{\mathbb{T}^3} n_{0,\delta} (1+|v_{0,\delta}|^2)\log{(1+|v_{0,\delta}|^2)}dx\leq C,
		\end{aligned}
	\end{equation}
	with $C>0$ a constant independent of $\delta$.

	We have the global well-posedness of strong solutions to the approximate problem (\ref{twoapp})-(\ref{dapp22}).
	\begin{prop}\label{appwell}
		Let $\delta\in(0,1)$, and the assumptions of Theorem \ref{theorem11} hold. There for suitably small $\var\in(0,\frac{1}{4})$, the IVP $(\ref{twoapp})$ has a unique strong solution $(n_{\varepsilon},v_{\varepsilon},\rho_{\varepsilon},u_{\varepsilon})$ satisfying for any $T>0$ that
		\begin{equation}\nonumber
			\left\{
			\begin{aligned}
				&\inf_{(x,t)\in \mathbb{T}^3\times[0,T]} n_{\varepsilon}(x,t)>0,\quad \quad \inf_{(x,t)\in \mathbb{T}^3\times[0,T]} \rho_{\varepsilon}(x,t)>0,\\
				&n_{\varepsilon},v_{\varepsilon},\rho_{\varepsilon},u_{\varepsilon}\in C([0,T]; H^2(\mathbb{T}^3))\cap L^2(0,T;H^3(\mathbb{T}^3)),\\
				& (n_{\varepsilon})_{t},(v_{\varepsilon})_{t},(\rho_{\varepsilon})_{t},(u_{\varepsilon})_{t} \in L^2(0,T;H^1(\mathbb{T}^3)).
			\end{aligned}
			\right.
		\end{equation}
	\end{prop}
	
	\vspace{2ex}

	The local well-posedness of the strong solution to the IVP $(\ref{twoapp})$ for $\var$, $\delta\in(0,1)$ can be proved in a standard way based on linearization techniques and fixed point arguments. We omit the proof here for brevity, and the reader can refer to for example \cite{mat1}. By virtue of the a-priori estimates established in Section \ref{section3} below, we are able to to extend the local approximate sequence $(n_{\varepsilon},v_{\varepsilon},\rho_{\varepsilon},u_{\varepsilon})$ to a global one and prove Proposition \ref{appwell}.

	\begin{remark}
	The approximate framework of  $(\ref{twoapp})_{1}$-$(\ref{twoapp})_{2}$ can be applied to the construction of approximate solutions for the three-dimensional isentropic compressible Navier-Stokes equations with constant viscosities.
	\end{remark}
	
	\section{The a-priori estimates}\label{section3}
	
	In this section we derive the uniformly a-priori estimates of approximate sequence given by Proposition \ref{appwell}. These estimates shall capture the basic energy, the Bresch-Desjardins type entropy, the Mellet-Vasseur type estimate, and the higher integrability of the density $\rho$, which are applied to obtain the upper and lower bounds of the densities by virtue of De Giorgi iteration, and then the higher-order estimates can be shown by standard regularity estimates for nonlinear parabolic equations.

	First of all, we have the basic energy estimates.
	\begin{lemma}\label{lemma31}
		For any $\var\in(0,\frac{1}{4})$, $\delta\in(0,1)$ and given time $T>0$, let $(n,v,\rho,u)$ be any strong solution to the IVP $(\ref{twoapp})$ for $t\in (0,T]$. Then, under the assumptions of Theorem \ref{theorem11}, we have  $n,\rho\geq0$ and 
			\begin{equation}\label{energyvar}
			\begin{aligned}
				&\sup_{t\in [0,T]}\int_{\mathbb{T}^3}\big{(}n|v|^2+n+\rho |u|^2+\rho^{\gamma}+\delta \rho^{\gamma_{0}} +\var n^{-12}\big{)}dx\\
				&\quad+\int_{0}^{T}\int_{\mathbb{T}^3} \big(\eta n|\mathbb{D}(v)|^2+\kappa n|v-u|^2+\mu |\nabla u|^2+(\mu+\lambda)(\dive u)^2 \big)dxdt\\
				&\quad+\int_{0}^{T}\int_{\mathbb{T}^3} \big(\sqrt{\var}n|\nabla v|^2+\var n|v|^{5}+\var(1+|v|^2)(|\nabla \sqrt{n}|^2+|\nabla \sqrt{n}|^4)+\var n^{-12}|v|^2+\var^2 n^{-25}\big)dxdt\\
				&\quad+\int_{0}^{T}\int_{\mathbb{T}^3} \var(|u|^{10}+|\nabla \rho|^2+ \rho^{\gamma-2}+\delta\rho^{\gamma_{0}-2}) dxdt\leq C_{T},
			\end{aligned}
		\end{equation}
		where $C_{T}>0$ is a constant independent of $\varepsilon$ and $\delta$.
	\end{lemma}
	
	\begin{proof}
	First, according to the maximum principle for the parabolic equations $(\ref{twoapp})_{1}$ and $(\ref{twoapp})_{3}$, both $n$ and $\rho$ are nonnegative. Then one obtains after integrating $(\ref{twoapp})_{1}$ and $(\ref{twoapp})_{3}$ over $\mathbb{T}^3$ that
					\begin{equation}
		\begin{aligned}
			&\frac{d}{dt}\int_{\mathbb{T}^3} n dx+\var \int_{\mathbb{T}^3} (|\nabla \sqrt{n}|^2+|\nabla \sqrt{n}|^4)dx=\var \int_{\mathbb{T}^3} n^{-12}dx,  \label{6} 
		\end{aligned}
		\end{equation}	
		and
		\begin{equation}\label{310}
		\begin{aligned}	
		&\frac{d}{dt}\int_{\mathbb{T}^3} \rho dx+\var\int_{\mathbb{T}^3}|\nabla \rho|^2dx=0.
		\end{aligned}
		\end{equation}	
	By $(\ref{twoapp})_{3}$-$(\ref{twoapp})_{4}$, we show 		
		\begin{equation}\label{3.4}
		\begin{aligned}
			&\frac{d}{dt}\int_{\mathbb{T}^3}\big{(}\frac{1}{2}\rho |u|^2+\frac{A\rho^{\gamma}}{\gamma-1}+\frac{\delta \rho^{\gamma_{0}}}{\gamma_{0}-1}\big{)}dx-\int_{\mathbb{T}^3} \big(\kappa n(v-u)\cdot u+\mu |\nabla u|^2+(\mu+\lambda)(\dive u)^2\big)dx  \\
			&\quad +\int_{\mathbb{T}^3}(\var(\gamma \rho^{\gamma-2}+\delta \gamma_{0}\rho^{\gamma_{0}-2})|\nabla \rho|^2+\var|u|^{10} \big)dx=0.    
					\end{aligned}
		\end{equation}	
		Meanwhile, we take the $L^2$-inner of $(\ref{twoapp})_2$ with $v$ to obtain
		\begin{equation}\label{3111}
		\begin{aligned}
			& \frac{d}{dt}\int_{\mathbb{T}^3} \frac{1}{2}n|v|^{2} dx+\int_{\mathbb{T}^3}\kappa n(v-u) \cdot vdx +\int_{\mathbb{T}^3} \eta n|\mathbb{D}(v)|^2dx   \\
			&\quad    +\int_{\mathbb{T}^3} (\var n|v|^{5}+\sqrt{\var}  n|\nabla v|^{2}+\frac{\var}{2} n^{-12}|v|^{2} +\frac{\var}{2} |\nabla\sqrt{n}|^2|v|^2+\frac{\var}{2}|\nabla\sqrt{n}|^4|v|^2)dx\\
			&\leq \int_{\mathbb{T}^3} n\dive v dx+\var\int_{\mathbb{T}^3} n|\nabla v|^2dx+\frac{\var}{4}\int_{\mathbb{T}^3}|\nabla\sqrt{n}|^2|v|^2dx.
		\end{aligned}
		\end{equation}	
	One deduces after adding (\ref{3.4})-(\ref{3111}) together that
		\begin{equation}\label{energyloss}
		\begin{aligned}
		&\frac{d}{dt}\int_{\mathbb{T}^3}\big{(} \frac{1}{2}n|v|^{2} +\frac{1}{2}\rho |u|^2+\frac{A\rho^{\gamma}}{\gamma-1}+\frac{\delta \rho^{\gamma_{0}}}{\gamma_{0}-1}\big{)}dx\\
		&\quad+\int_{\mathbb{T}^3}\big(\kappa n|v-u|^2+\eta n|\mathbb{D}(v)|^2+\mu |\nabla u|^2+(\mu+\lambda)(\dive u)^2\big)dx\\
		&\quad+\int_{\mathbb{T}^3}\big(\var n|v|^{5} +(\sqrt{\var}-\var)  n|\nabla v|^{2}+\frac{\var}{2} n^{-12}|v|^{2} +\frac{\var}{4} |\nabla\sqrt{n}|^2|v|^2+\frac{\var}{2}|\nabla\sqrt{n}|^4|v|^2 \big)dx\\
		&\quad+\int_{\mathbb{T}^3}\big(\var |u|^{10}+\var(\gamma \rho^{\gamma-2}+\delta \gamma_{0}\rho^{\gamma_{0}-2})|\nabla \rho|^2\big)dx \le \int_{\mathbb{T}^3} n\dive v dx.  
		\end{aligned}
		\end{equation}		
We would like to mention that if we estimate the usual term $n\log{n}$ to cancel the first term on the right-hand side of (\ref{3111}), then the diffusion terms in $(\ref{twoapp})_{1}$ will cause an additional difficulty since $1+\log{n}$ may not be positive. To overcome this difficulty, we have by Young$'$s inequality and the fact $|\dive v|^2\leq 3|\mathbb{D}(v)|^2$ that
\begin{equation}\label{312}
\begin{aligned}
&\int_{\mathbb{T}^3} n\dive v dx\leq \int_{\mathbb{T}^3} n\chi(n)\dive v dx+\frac{\eta}{4} \int_{\mathbb{T}^3}n|\mathbb{D}(v)|^2 dx+\frac{C}{\eta},
\end{aligned}
\end{equation}
where $\chi(s)\geq0$ is a smooth function on $\mathbb{R}_{+}$ satisfying $\chi(s)=0$ for $s\in[0,1]$, $\chi(s)=1$ for $s\geq e$ and $3\chi(s)+2s\chi'(s)\geq0$ for $s\geq0$. Denoting $\Pi(n)=n\int_{0}^{n}s^{-1}\chi(s)ds$, we get from  $(\ref{twoapp})_{1}$ and $\Pi'(n)+2n\Pi''(n)=\int_{0}^{n} s^{-1} \chi(s) ds +3 \chi(n)+2n\chi'(n)\geq 0$ that 
		\begin{equation}\label{3121}
		\begin{aligned}
		\frac{d}{dt} \int_{\mathbb{T}^3} \Pi(n)dx &\le \frac{d}{dt} \int_{\mathbb{T}^3} \Pi(n)dx+\var\int_{\mathbb{T}^3} (\Pi'(n)+2n\Pi''(n)) |\nabla\sqrt{n}|^2(1+|\nabla\sqrt{n}|^2) dx \\
			&=-\int_{\mathbb{T}^3} n\chi(n)\dive v dx+\var\int_{\mathbb{T}^3} \Pi'(n) n^{-12}dx\\
			&\leq -\int_{\mathbb{T}^3} n\chi(n)\dive v dx+\frac{\var^2}{26} \int_{\mathbb{T}^3} n^{-25}dx +C,
		\end{aligned}
		\end{equation}
		where one has used $|\Pi^{'}(s)|\leq C(1+s)$ and the simple fact
		\begin{align}
 n^{-q}\leq C(n^{-25}+1),\quad\quad q\in(0,25).\label{factpp}
\end{align}
		In addition, according to $(\ref{twoapp})_{1}$ and $|\dive v|^2\leq 3|\mathbb{D}(v)|^2$, we also get 
		      \begin{equation}\label{7}
		\begin{aligned}
		&\frac{d}{dt} \int_{\mathbb{T}^3} \frac{\var}{78}n^{-12}dx+\int_{\mathbb{T}^3} (\frac{50\var^2}{13}  n^{-13} |\nabla \sqrt{n}|^2(1+|\nabla \sqrt{n}|^2)+\frac{2\var^2}{13} n^{-25})dx\\
		&=\frac{\var}{6}\int_{\mathbb{T}^3}n^{-12} \dive v dx\leq  \frac{\var^2}{13}\int_{\mathbb{T}^3} n^{-25} dx+\frac{1}{2}\int_{\mathbb{T}^3}n|\mathbb{D}(v)|^2dx.
	       \end{aligned}
		\end{equation}
		The combination of $(\ref{6})$-$(\ref{7})$ and the Gr${\rm\ddot{o}}$nwall inequality leads to $(\ref{energyvar})$.
	\end{proof}
	
	\begin{remark}
 We are able to recover the energy inequality uniformly in time after taking the limit as $\var\rightarrow0$, see the proof of Proposition \ref{weakdelta}.
	\end{remark}

	Next, in order to obtain the uniform spatial derivative estimate of $\sqrt{n}$, we show a Bresch-Desjardins type entropy estimates.
	\begin{lemma}\label{lemma32}
		For any $\var\in(0,\frac{1}{4})$, $\delta\in(0,1)$ and given time $T>0$, let $(n,v,\rho,u)$ be any strong solution to the IVP $(\ref{twoapp})$ for $t\in (0,T]$. Then, under the assumptions of Theorem \ref{theorem11}, it holds
		  \begin{equation}\label{BDvar}
		\begin{aligned}
			&\sup_{t\in[0,T]}\int_{\mathbb{T}^3} (|\nabla \sqrt{n}|^2+\var|\nabla \sqrt{n}|^4)dx+\int_{0}^T\int_{\mathbb{T}^3} (n|\mathbb{A}(v)|^2+n|\nabla v|^2+|\nabla \sqrt{n}|^2)dxdt   \\
			&\quad\quad+\var\int_{0}^T\int_{\mathbb{T}^3} (|\nabla^2 \sqrt{n}|^2+|\nabla \sqrt{n}|^2|\nabla^2 \sqrt{n}|^2)dxdt+\var^2\int_{0}^T\int_{\mathbb{T}^3} |\nabla \sqrt{n}|^4|\nabla^2 \sqrt{n}|^2dxdt  \le C_{T},
	        \end{aligned}
		\end{equation}
		where $\mathbb{A}(v):=\frac{\nabla v-(\nabla v)^{tr}}{2}$ denotes the antisymmetric part of the gradient, and $C_{T}>0$ is a constant independent of $\varepsilon$ and $\delta$.
	\end{lemma}
	
	\begin{proof}
	Let $c_{0}>0$ be a constant to be chosen later. We denote the effective velocity
	$$
	w:=v+c_{0}\nabla\log n
	$$ 
	to rewrite the mass equation $(\ref{twoapp})_1$ by
	\begin{equation}\nonumber
	\begin{aligned}
	&n_{t}+\dive (nw)=c_{0}\Delta n+G,\quad\quad G:=\var \sqrt{n} \Delta \sqrt{n}+\var \sqrt{n} \dive (|\nabla \sqrt{n}|^2 \nabla \sqrt{n})+\var n^{-12},
	\end{aligned}
	\end{equation}
	and the momentum equation $(\ref{twoapp})_2$ by 
	\begin{equation}\nonumber
	\begin{aligned}
	&(nw)_{t}+\dive(nw\otimes w)+\nabla n-c_{0}\Delta(nw)+\var n|v|^3v+\var n^{-12} v\\
	&\quad\quad\quad=(\eta+\sqrt{\var}-2c_{0})\dive (n\mathbb{D}w)+\sqrt{\var}\dive (n\mathbb{A}(v))-(c_{0}^2+(\eta+\sqrt{\var}-2c_{0})c_{0})\dive (n\nabla^2\log{n})\\
	&\quad\quad\quad\quad+\nabla G+G v+\var\sqrt{n}|\nabla\sqrt{n}|^2\nabla\sqrt{n}\cdot\nabla v,
	\end{aligned}
	\end{equation}
	where one has used the following equalities:
	\begin{equation}\nonumber
	\left\{
	\begin{aligned}
	&(n\nabla\log{n})_{t}=-\nabla\dive (nv)+\nabla G,\\
	&\dive (n v\otimes \nabla\log{n}+n\nabla\log{n}\otimes v)=\Delta(nv)-2\dive (n\mathbb{D}v)+\nabla\dive (nv),\\
	&\dive(n\nabla\log{n}\otimes\nabla\log{n})=\Delta(n\nabla\log{n})-\dive (n\nabla^2\log{n}).
	\end{aligned}
	\right.
	\end{equation}
	Thus, choosing 
	$$
	c_{0}=\eta+\sqrt{\var},
	$$
	we derive the following reformulated equations:
	\begin{equation}\label{reformula1}
	\left\{
	\begin{aligned}
	&n_{t}+\dive (nw)=c_{0}\Delta n+G,\\
	&n(w_{t}+w\cdot \nabla w)+\nabla n+\kappa n(v-u)+c_{0}w\Delta n-c_{0}\Delta(nw)+\var n|v|^3v+\var n^{-12} v \\
	&\quad\quad\quad=-c_{0}\dive (n\mathbb{D}w)+\sqrt{\var}\dive (n\mathbb{A}(v))+\nabla G-c_{0} G\nabla \log{n}+\var\sqrt{n}|\nabla\sqrt{n}|^2\nabla\sqrt{n}\cdot \nabla v.
	\end{aligned}
	\right.
	\end{equation}
	Thence we obtain after taking the $L^2(\mathbb{T}^3)$-inner product of $(\ref{reformula1})_{2}$ with $w$ and making use of $(\ref{reformula1})_{1}$ that
		\begin{equation}\label{reformula111}
	\begin{aligned}
	&\frac{d}{dt}\int_{\mathbb{T}^3}\frac{1}{2}n|w|^2dx+\int_{\mathbb{T}^3} ((c_{0}+\sqrt{\var})n|\mathbb{A}(v)|^2+2c_{0}|\nabla\sqrt{n}|^2+\nabla n\cdot v+\kappa n(v-u)v)dx \\
	&= -c_{0}\int_{\mathbb{T}^3}(\kappa n(v-u) +\var n|v|^3v)\cdot\nabla \log{n}dx-\int_{\mathbb{T}^3} G\dive vdx-c_{0}\int_{\mathbb{T}^{3} }G\Delta\log{n} dx\\
	&\quad+\int_{\mathbb{T}^3} G (\frac{1}{2}|w|^2-c_{0}\nabla\log{n}\cdot w)  dx+\var \int_{\mathbb{T}^3}	\sqrt{n}|\nabla\sqrt{n}|^2\nabla\sqrt{n}\cdot\nabla v\cdot w dx-\var\int_{\mathbb{T}^3} n^{-12} v\cdot w dx.
		\end{aligned}
	\end{equation}
We estimate the terms on the right-hand side of (\ref{reformula111}) as follows. First, we can show after using H${\rm{\ddot{o}}}$lder's and Young's inequalities that
	\begin{equation}\label{I2}
	\begin{aligned}
	&-c_{0}\int_{\mathbb{T}^3}(\kappa n(v-u)+\var n|v|^3v)\cdot \nabla \log{n}dx\\
	&\quad\leq \kappa c_{0} \big(\int_{\mathbb{T}^3} n|v-u|^2dx\big)^{\frac{1}{2}}\big(\int_{\mathbb{T}^3}|\nabla \sqrt{n}|^2dx\big)^{\frac{1}{2}}\\
	&\quad\quad+2c_{0}\var \big( \int_{\mathbb{T}^3} n|v|^{5}dx\big)^{\frac{7}{10}} \big( \int_{\mathbb{T}^3} |\nabla\sqrt{n}|^4|v|^2dx\big)^{\frac{1}{4}} \big(\int_{\mathbb{T}^{3}} n^{-4}dx\big)^{\frac{1}{20}}\\
	&\quad\leq c_{0} \int_{\mathbb{T}^3}|\nabla \sqrt{n}|^2dx+\var\int_{\mathbb{T}^3} n|v|^5 dx+\frac{\var}{2}\int_{\mathbb{T}^3}|\nabla \sqrt{n}|^4 |v|^2dx\\
	&\quad\quad+C\int_{\mathbb{T}^3} n|v-u|^2dx+C\var \int_{\mathbb{T}^3} n^{-25}dx+C\var.
		\end{aligned}
	\end{equation}	
Similarly, one has	
		\begin{equation}\label{I3}
	\begin{aligned}
	&-\int_{\mathbb{T}^3}G \dive udx\\
	&\quad=-\var\int_{\mathbb{T}^3} \sqrt{n}\big (\Delta\sqrt{n}+\dive (|\nabla \sqrt{n}|^2\nabla\sqrt{n}) \big) \dive u dx+\var\int_{\mathbb{T}^3}n^{-12} \dive udx\\
	&\quad\leq \frac{\var^2}{8} \int_{\mathbb{T}^3}\big (\Delta\sqrt{n}+\dive (|\nabla \sqrt{n}|^2\nabla\sqrt{n}) \big) ^2dx+C\int_{\mathbb{T}^3} n|\dive u|^2dx+C\var^2\int_{\mathbb{T}^3} n^{-25}dx.
			\end{aligned}
	\end{equation}
And as the property
\begin{equation}\nonumber
	\begin{aligned}
	\int_{\mathbb{T}^3}\dive(|\nabla\sqrt{n}|^2\nabla\sqrt{n})\Delta\sqrt{n}dx&=-\int_{\mathbb{T}^3}|\nabla\sqrt{n}|^2\nabla\sqrt{n}\cdot\nabla\Delta\sqrt{n}dx  \notag  \\
	&=\int_{\mathbb{T}^3}|\nabla\sqrt{n}|^2|\nabla^2\sqrt{n}|^2dx+\frac{1}{2}\int_{\mathbb{T}^3}|\nabla|\nabla\sqrt{n}|^2|^2dx,
	\end{aligned}
	\end{equation}
the fourth term on the right-hand side of (\ref{reformula111}) can be decomposed by
	\begin{equation}\label{3.16}
	\begin{aligned}
	&-c_{0}\int_{\mathbb{T}^{3} }G\Delta\log{n} dx\\
	&\quad=-2c_{0}\var\int_{\mathbb{T}^3} \sqrt{n}(\Delta\sqrt{n}+\dive (|\nabla \sqrt{n}|^2\nabla\sqrt{n}) )(n^{-\frac{1}{2}}\Delta\sqrt{n}-n^{-1}|\nabla \sqrt{n}|^2)dx\\
	&\quad\quad-\var c_{0}\int_{\mathbb{T}^3}n^{-12}\Delta\log{n}dx\\
	&\quad=-\var c_{0}\int_{\mathbb{T}^3}( 2|\Delta \sqrt{n}|^2+2|\nabla\sqrt{n}|^2|\nabla^2\sqrt{n}|^2+|\nabla|\nabla\sqrt{n}|^2|^2+48n^{-13}|\nabla \sqrt{n}|^2)dx\\
	&\quad\quad+2 c_{0}\var\int_{\mathbb{T}^3}(\Delta\sqrt{n}+\dive (|\nabla \sqrt{n}|^2\nabla\sqrt{n}))n^{-\frac{1}{2}}|\nabla \sqrt{n}|^2dx.
	\end{aligned}
	\end{equation}
	Note that the Sobolev inequality implies that
	\begin{equation}
		\left\{
	\begin{aligned}
		&\| \nabla\sqrt{n}\|_{L^{6}} \le C \| \nabla^2 \sqrt{n}\|_{L^{2}} \le C \| \Delta \sqrt{n}\|_{L^{2}} \notag \\
	&\| |\nabla\sqrt{n}|^2\|_{L^{4}}\leq \| |\nabla\sqrt{n}|^2\|_{L^{2}}^{\frac{1}{4}} \| |\nabla\sqrt{n}|^2\|_{L^{6}}^{\frac{3}{4}}\leq C\| \nabla\sqrt{n}\|_{L^{2}}^{\frac{1}{2}}\|\nabla |\nabla \sqrt{n}|^2\|_{L^2}^{\frac{3}{4}}+\| \nabla\sqrt{n}\|_{L^{2}},  \notag
	\end{aligned}
    \right.
	\end{equation}
which together with (\ref{factpp}) gives rise to
	\begin{equation}\label{therefore}
	\begin{aligned}
	&\var\int_{\mathbb{T}^3}(\Delta\sqrt{n}+\dive (|\nabla \sqrt{n}|^2\nabla\sqrt{n}))n^{-\frac{1}{2}}|\nabla \sqrt{n}|^2dx\\
	&\quad\leq \var \| \Delta \sqrt{n}\|_{L^{2}}\| |\nabla\sqrt{n}|^2\|_{L^{4}}\|n^{-\frac{1}{2}}\|_{L^{4}}+3\var\|\nabla\sqrt{n}\Delta \sqrt{n}\|_{L^{2}}\| |\nabla\sqrt{n}|^2\|_{L^{4}}\| \nabla\sqrt{n}\|_{L^{6}}\|n^{-\frac{1}{2}}\|_{L^{12}} \\
	&\quad\leq  \frac{c_{0}\var}{4}\int_{\mathbb{T}^3}( 2|\Delta \sqrt{n}|^2+2|\nabla\sqrt{n}|^2|\nabla^2\sqrt{n}|^2+|\nabla|\nabla\sqrt{n}|^2|^2+48n^{-13}|\nabla \sqrt{n}|^2)dx\\
	&\quad\quad+C\var\int_{\mathbb{T}^3}(|\nabla\sqrt{n}|^4+|\nabla\sqrt{n}|^2+n^{-25})dx+C\var.
	\end{aligned}
	\end{equation}
	Therefore, we have
		\begin{equation}\label{I4}
	\begin{aligned}
	&-c_{0}\int_{\mathbb{T}^{3} }G\Delta\log{n} dx\\
	&\quad\leq-\frac{3\var c_{0}}{4}\int_{\mathbb{T}^3}( 2|\Delta \sqrt{n}|^2+2|\nabla\sqrt{n}|^2|\nabla^2\sqrt{n}|^2+|\nabla|\nabla\sqrt{n}|^2|^2+48n^{-13}|\nabla \sqrt{n}|^2)dx \\
	&\quad\quad+C\var\int_{\mathbb{T}^3}(|\nabla\sqrt{n}|^4+|\nabla\sqrt{n}|^2+n^{-25})dx+C\var.
		\end{aligned}
	\end{equation}	
Finally, by \eqref{therefore} and integration by parts, we have
		\begin{equation}\label{3190}
	\begin{aligned}	
	&\int_{\mathbb{T}^3} G(\frac{1}{2}|w|^2- c_{0}\nabla\log{n}\cdot w)  dx+\var \int_{\mathbb{T}^3}	\sqrt{n}|\nabla\sqrt{n}|^2\nabla\sqrt{n}\cdot\nabla v\cdot w dx-\var\int_{\mathbb{T}^3} n^{-12} v\cdot w dx\\
	&\quad=-\int_{\mathbb{T}^3} (\var n|v|^{5}+\sqrt{\var}  n|\nabla v|^{2}+\frac{\var}{2} n^{-12}|v|^{2} +\frac{\var}{2} |\nabla\sqrt{n}|^2|v|^2+\frac{\var}{2}|\nabla\sqrt{n}|^4|v|^2)dx\\
	&\quad\quad+2c_{0}^2\var\int_{\mathbb{T}^3}(\Delta\sqrt{n}+\dive (|\nabla \sqrt{n}|^2\nabla\sqrt{n}))n^{-\frac{1}{2}}|\nabla \sqrt{n}|^2dx\\
		&\quad\quad-2\var c_{0}\int_{\mathbb{T}^3}v\cdot \dive (|\nabla\sqrt{n}|^2\nabla\sqrt{n}\otimes\nabla\sqrt{n})dx+\var\int_{\mathbb{T}^3} n^{-12}(\frac{1}{2} |v|^2+\frac{c_{0}^2}{2}|\nabla\log{n}|^2-v\cdot w)dx\\
	&\quad\leq  \frac{\var c_{0}}{4}\int_{\mathbb{T}^3}( 2|\Delta \sqrt{n}|^2+2|\nabla\sqrt{n}|^2|\nabla^2\sqrt{n}|^2+|\nabla|\nabla\sqrt{n}|^2|^2+48n^{-13}|\nabla \sqrt{n}|^2)dx    \\
	&\quad\quad+C\var\int_{\mathbb{T}^3}(|\nabla\sqrt{n}|^4|v|^2+|\nabla\sqrt{n}|^4+|\nabla\sqrt{n}|^2+n^{-25})dx+C\var.
	\end{aligned}
	\end{equation}
	Substituting the above estimates (\ref{I2})-(\ref{3190}) into (\ref{reformula1}) and using (\ref{factpp}), we get
	\begin{equation}\label{reformula11}
	\begin{aligned}
	&\frac{d}{dt}\int_{\mathbb{T}^3}\frac{1}{2}n|w|^2dx+\int_{\mathbb{T}^3} \big((c_{0}+\sqrt{\var})n|\mathbb{A}(v)|^2+c_{0}|\nabla\sqrt{n}|^2+\nabla n\cdot v+\kappa n(v-u)v \big)dx\\
	&\quad+\frac{\var c_{0}}{2}\int_{\mathbb{T}^3}( |\Delta \sqrt{n}|^2+|\nabla\sqrt{n}|^2|\nabla^2\sqrt{n}|^2+\frac{1}{2}|\nabla|\nabla\sqrt{n}|^2|^2+24n^{-13}|\nabla\sqrt{ n}|^2)dx\\
	&\leq \frac{\var^2}{8} \int_{\mathbb{T}^3}\big (\Delta\sqrt{n}+\dive (|\nabla \sqrt{n}|^2\nabla\sqrt{n}) \big) ^2dx\\
	&\quad+C\int_{\mathbb{T}^3}( n|v-u|^2+\var n|v|^5+\var|\nabla \sqrt{n}|^4+\var|\nabla \sqrt{n}|^2+\var |\nabla\sqrt{n}|^4|v|^2+\var n^{-25})dx+C\var.
		\end{aligned}
	\end{equation}

Furthermore, by $(\ref{twoapp})_{1}$, the equation of $\sqrt{n}$ reads
		\begin{align}
			2(\sqrt{n})_{t}-\var \Delta \sqrt{n}-\var \dive (|\nabla \sqrt{n}|^2 \nabla \sqrt{n})=-2v \cdot \nabla \sqrt{n}-\sqrt{n} \dive v+\var \sqrt{n}^{-25} .          \label{211}
		\end{align}
		One multiplies $(\ref{211})$ by $\var(\Delta \sqrt{n}+\dive (|\nabla \sqrt{n}|^2 \nabla \sqrt{n}))$ and integrates the resulted equation by parts to get
		\begin{align}
			&\frac{d}{dt}\int_{\mathbb{T}^3}( \var |\nabla \sqrt{n}|^2 +\frac{\var}{2} |\nabla \sqrt{n}|^4 )dx+\var^2 \int_{\mathbb{T}^3} (\Delta \sqrt{n}+\dive (|\nabla \sqrt{n}|^2 \nabla \sqrt{n}))^2 dx  \notag   \\
			&\quad+25\var^2\int_{\mathbb{T}^3} n^{-13}(|\nabla \sqrt{n}|^2+|\nabla \sqrt{n}|^4)dx   \notag  \\
			&=\var \int_{\mathbb{T}^3} (\Delta \sqrt{n}+\dive (|\nabla \sqrt{n}|^2 \nabla \sqrt{n})) \sqrt{n}\dive v dx+2\var \int_{\mathbb{T}^3} (\Delta \sqrt{n}+\dive (|\nabla \sqrt{n}|^2 \nabla \sqrt{n}))v \cdot \nabla \sqrt{n}dx   \notag  \\
			&\le \frac{\var^2}{8}\int_{\mathbb{T}^3} (\Delta \sqrt{n}+\dive (|\nabla \sqrt{n}|^2 \nabla \sqrt{n}))^2dx+\frac{\var c_{0}}{8} \int_{\mathbb{T}^3} (|\Delta \sqrt{n}|^2+|\nabla \sqrt{n}|^2|\nabla^2 \sqrt{n}|^2)dx  \notag   \\
			&\quad+C\int_{\mathbb{T}^3} n(\dive v)^2dx+C\var\int_{\mathbb{T}^3} |v|^2|\nabla \sqrt{n}|^2(1+|\nabla \sqrt{n}|^2)dx,  \nonumber
		\end{align}
		which together with (\ref{reformula11}) and the fact
		\begin{align}
		&\int_{\mathbb{T}^3} (\Delta \sqrt{n}+\dive (|\nabla \sqrt{n}|^2 \nabla \sqrt{n}))^2 dx\nonumber\\
		&\geq \frac{1}{2}\int_{\mathbb{T}^3} \big(\dive (|\sqrt{n}|^2 \nabla \sqrt{n}))^2dx-2\int_{\mathbb{T}^3} (\Delta \sqrt{n})^2dx\nonumber\\
		&=\frac{1}{2}\int_{\mathbb{T}^3}  (|\nabla \sqrt{n}|^4|\nabla^2 \sqrt{n}|^2+(\nabla \sqrt{n}\cdot \nabla |\nabla \sqrt{n}|^2)^2+|\nabla \sqrt{n}|^2|\nabla |\nabla \sqrt{n}|^2|^2 \big)dx-2\int_{\mathbb{T}^3} (\Delta \sqrt{n})^2dx\nonumber
		\end{align} 
		leads to 
\begin{equation}\label{BDii}
\begin{aligned}
&\frac{d}{dt}\int_{\mathbb{T}^3}(\frac{1}{2}n|w|^2 +\var |\nabla \sqrt{n}|^2 +\frac{\var}{2} |\nabla \sqrt{n}|^4 )dx\\
&\quad+\int_{\mathbb{T}^3} ((\eta+2\sqrt{\var})n|\mathbb{A}(v)|^2+c_{0}|\nabla\sqrt{n}|^2+\nabla n\cdot v+\kappa n(v-u)v\big)dx\\
&\quad+\frac{\var}{4}(c_{0}-2\var)\int_{\mathbb{T}^3}( |\Delta \sqrt{n}|^2+|\nabla\sqrt{n}|^2|\nabla^2\sqrt{n}|^2+\frac{1}{2}|\nabla|\nabla\sqrt{n}|^2|^2+24 n^{-13}|\nabla n|^2)dx\\
&\quad+\frac{\var^2}{4}\int_{\mathbb{T}^3}( |\nabla \sqrt{n}|^4|\nabla^2 \sqrt{n}|^2+(\nabla \sqrt{n}\cdot \nabla |\nabla \sqrt{n}|^2)^2+|\nabla \sqrt{n}|^2|\nabla |\nabla \sqrt{n}|^2|^2 \big)dx\\
&\quad+25\var^2\int_{\mathbb{T}^3} n^{-13}(|\nabla \sqrt{n}|^2+|\nabla \sqrt{n}|^4)dx  \\
&\leq C\int_{\mathbb{T}^3}( n(\dive v)^2+n|v-u|^2)dx\\
&\quad+\var\int_{\mathbb{T}^3}( n|v|^5+|\nabla \sqrt{n}|^4+|\nabla \sqrt{n}|^2+ |v|^2|\nabla \sqrt{n}|^2(1+|\nabla \sqrt{n}|^2)+ n^{-25})dx+C\var.
\end{aligned}
\end{equation}
The combination of (\ref{energyvar}), (\ref{3.4}),  (\ref{312})-(\ref{3121}) and (\ref{BDii}) gives rise to (\ref{BDvar}). The proof of Lemma \ref{lemma32} is completed.
	\end{proof}
	
	\begin{remark}
	We will show that the Bresch-Desjardins type estimates do not depend on time after the limit process as $\var\rightarrow0$, see the proof of Proposition \ref{weakdelta}.
	\end{remark}

	The following Mellet-Vasseur type estimate will be used to prove the strong convergence of $\sqrt{n}v$ in $L^2(0,T;L^2(\mathbb{T}^3))$.
	\begin{lemma}\label{lemma33}
		For any $\var\in(0,\frac{1}{4})$, $\delta\in(0,1)$ and given time $T>0$, let $(n,v,\rho,u)$ be any strong solution to the IVP $(\ref{twoapp})$ for $t\in (0,T]$. Then, under the assumptions of Theorem \ref{theorem11}, it holds
		\begin{equation}\label{melletvar}
			\begin{aligned}
				&\sup_{t\in[0,T]}\int_{\mathbb{T}^3} n(1+|v|^2)\log{(1+|v|^2)}dx\leq C_{T},
			\end{aligned}
		\end{equation}
		where $C_{T}>0$ is a constant independent of $\varepsilon$ and $\delta$.
	\end{lemma}
	\begin{proof}
		Multiplying $(\ref{twoapp})_{2}$ by $(1+\log (1+|v|^{2}))v$, we  obtain after integration by parts that
		\begin{equation}\label{11}
			\begin{aligned}
			&\frac{d}{dt} \int_{\mathbb{T}^3} \frac{n(1+|v|^{2})}{2}\log(1+|v|^{2})dx+\int_{\mathbb{T}^3} n\log(1+|v|^{2})(\eta |\mathbb{D}(v)|^2+\sqrt{\var}|\nabla v|^{2})dx  \\
			&\quad+\var\int_{\mathbb{T}^3}  n |v|^{5}(1+\log(1+|v|^{2}))dx+\frac{\var}{2} \int_{\mathbb{T}^3} (1+|v|^2) \log(1+|v|^2) |\nabla \sqrt{n}|^4 dx   \\ 
			&\quad+\kappa\int_{\mathbb{T}^3}  n|v|^2(1+\log(1+|v|^{2}))dx  \\
			&=\frac{\var}{2}\int_{\mathbb{T}^3} (1+|v|^2)\sqrt{n}\Delta \sqrt{n} \log(1+|v|^{2}) dx +\frac{\var}{2}\int_{\mathbb{T}^3} n^{-12} (1+v^{2})\log(1+|v|^{2})dx   \\
			&\quad-\var\int_{\mathbb{T}^3} n^{-12} |v|^{2}(1+\log(1+|v|^{2}))dx+\int_{\mathbb{T}^3} v \cdot \nabla n (1+\log(1+|v|^{2}))dx    \\
			&\quad+\kappa\int_{\mathbb{T}^3}  n u\cdot v(1+\log(1+|v|^{2}))dx. 
			\end{aligned}
		\end{equation}
		The first term on right-hand side of $(\ref{11})$ can be estimated as 
		\begin{equation}\label{I_1}
			\begin{aligned}
			&\frac{\var}{2}\int_{\mathbb{T}^3} (1+|v|^2)\sqrt{n}\Delta \sqrt{n} \log(1+|v|^{2}) dx     \\
			&\le -\frac{\var}{8}\int_{\mathbb{T}^3} |\nabla \sqrt{n}|^2(1+|v|^2)\log(1+|v|^2)dx  +\frac{\var}{2}\int_{\mathbb{T}^3} n \log(1+|v|^2)|\nabla v|^2dx   \\
			&\quad +C\var\int_{\mathbb{T}^3} |\nabla \sqrt{n}|^2 |v|^2dx +C\var\int_{\mathbb{T}^3} n|\nabla v|^2 dxx.  
			\end{aligned}
		\end{equation}
	In addition, one has
	\begin{align}
		\frac{\var}{2}\int_{\mathbb{T}^3} n^{-12} (1+|v|^{2})\log(1+|v|^{2})dx-\var\int_{\mathbb{T}^3} n^{-12} |v|^{2}(1+\log(1+|v|^{2}))dx  \le C\var\int_{\mathbb{T}^3} n^{-25}dx+C\var,  \notag
	\end{align}
and
\begin{align}
	\int_{\mathbb{T}^3} v \cdot \nabla n (1+\log(1+|v|^{2}))dx \le C\int_{\mathbb{T}^3} n|v|^2dx+C\int_{\mathbb{T}^3} n|\nabla v|^2. \notag
\end{align}
For the last term on the right-hand side of $(\ref{11})$, we make use of \eqref{energyvar}, \eqref{BDvar} and the fact $\log(1+|v|^2)\le C|v|^{\frac{1}{4}}$ that 
\begin{equation}\label{kappa}
	\begin{aligned}
	&\kappa\int_{\mathbb{T}^3}  n u \cdot v(1+\log(1+|v|^{2}))dx  \\
	&\le C \|u\|_{L^6}(\|\sqrt{n}\|_{L^3}\|\sqrt{n}v\|_{L^2}+\|n^{\frac{3}{8}}\|_{L^{\frac{24}{5}}}\|n^{\frac{5}{8}}v^{\frac{5}{4}}\|_{L^{\frac{8}{5}}})  \\
	&\le C\|\nabla u\|_{L^2}(\|\sqrt{n}\|_{H^1}\|\sqrt{n}v\|_{L^2}+\|\sqrt{n}\|_{H^1}^{\frac{3}{4}}\|\sqrt{n}v\|^{\frac{5}{4}}_{L^2})  \\
	&\le C\int_{\mathbb{T}^3} |\nabla u|^2dx+C.
	\end{aligned}
\end{equation}
		Inserting the above estimates $(\ref{I_1})$-$(\ref{kappa})$ into $(\ref{11})$ and applying the Gr${\rm\ddot{o}}$nwall inequality, we gain $(\ref{melletvar})$. The proof of Lemma \ref{lemma33} is completed.
	\end{proof}

	Then, we establish the uniform $L^{\gamma+1}(0,T;L^{\gamma+1}(\mathbb{T}^3))$ estimate of $\rho$ so as to show the strong convergence of the pressure.
	\begin{lemma}\label{lemma34}
		For any $\var\in(0,\frac{1}{4})$, $\delta\in(0,1)$ and given time $T>0$, let $(n,v,\rho,u)$ be any strong solution to the IVP $(\ref{twoapp})$ for $t\in (0,T]$. Then, under the assumptions of Theorem \ref{theorem11}, we have
		\begin{equation}\label{gamma1}
			\begin{aligned}
				&\int_{0}^{T}\int_{\mathbb{T}^3} (\rho^{\gamma+1}+\delta \rho^{\gamma_{0}+1})dxdt\leq C_{\delta,T},
			\end{aligned}
		\end{equation}
		where $C_{\delta,T}>0$ is a constant independent of $\varepsilon$.
	\end{lemma}
	
	\begin{proof}	
	Denote the operator $(-\Delta)^{-1}: W^{k-2,p}(\mathbb{T}^{d}) \rightarrow W^{k,p}(\mathbb{T}^{d})$ for $k\in\mathbb{R}$ and $p\in(1,\infty)$ by
\begin{equation}\label{deltam1}
\begin{aligned}
&(-\Delta)^{-1} f=g,\quad g~\text{is the solution of the problem}~-\Delta g=f,\quad \int_{\mathbb{T}^{3}}f=0.
\end{aligned}
\end{equation}
One obtains after applying $(-\Delta)^{-1}\dive$ to $(\ref{twoapp})_{4}$ that
\begin{equation}\label{rhoformular}
\begin{aligned}
&A\rho^{\gamma}+\delta \rho^{\gamma_{0}}-(2\mu+\lambda)\dive u\\
&\quad=(-\Delta)^{-1}\dive(\rho u)_{t}+(-\Delta)^{-1}\dive\dive(\rho u\otimes u)+\kappa (-\Delta)^{-1}\dive (n(u-v))\\
&\quad\quad+\var (-\Delta)^{-1}\dive\dive (\nabla \rho\otimes u)+\var (-\Delta)^{-1}\dive (|u|^{8}u).
\end{aligned}
\end{equation}
Then we multiply (\ref{rhoformular}) by $\rho$ and integrate the resulted equality over $\mathbb{T}^{3}\times(0,T)$ to gain
\begin{equation}\label{rhoformular11}
\begin{aligned}
&\int_{0}^{T}\int_{\mathbb{T}^3}(A\rho^{\gamma+1}+\delta \rho^{\gamma_{0}+1})dxdt\\
&=\int_{\mathbb{T}^3}\rho (-\Delta)^{-1}\dive(\rho u) dx\big|^{t=T}_{t=0}\\
&\quad+\int_{0}^{T}\int_{\mathbb{T}^3}\rho [ (-\Delta)^{-1}\dive \dive (\rho u\otimes u)-u\nabla(-\Delta)^{-1}\dive (\rho u)]dxdt\\
&\quad+\kappa \int_{0}^{T}\int_{\mathbb{T}^3} \rho (-\Delta)^{-1}\dive (n(u-v)) dxdt\\
&\quad+\var\int_{0}^{T}\int_{\mathbb{T}^3} \big( \nabla\rho \cdot \nabla (-\Delta)^{-1}\dive (\rho u)+\rho (-\Delta)^{-1}\dive\dive (\nabla \rho\otimes u) +\rho  (-\Delta)^{-1}\dive (|u|^{8}u)\big)dxdt.
\end{aligned}
\end{equation}
It can be verified by (\ref{energyvar}), (\ref{BDvar}), the Sobolev inequality and $L^{p}(\mathbb{T}^{3})$ ($p\in(1,\infty)$) boundedness of the operator $(-\Delta)^{-1} \partial_{i}\partial_{j} $ for $1\leq i,j\leq 3$ that
\begin{equation}\nonumber
\begin{aligned}
&\kappa\int_{0}^{T}\int_{\mathbb{T}^3} \rho (-\Delta)^{-1}\dive (n(u-v)) dxdt\\
&\quad\leq \|\rho \|_{L^2(0,T;L^2)} \|n(u-v)\|_{L^2(0,T;L^{\frac{6}{5}}(\mathbb{T}^3))}\\
&\quad\leq CT^{\frac{1}{2}}\|\rho \|_{L^{\infty}(0,T;L^{\gamma_{0}}(\mathbb{T}^3))} \|\sqrt{n}\|_{L^{\infty}(0,T;L^{2}(\mathbb{T}^3))}\|\sqrt{n}(u-v)\|_{L^2(0,T;L^{2}(\mathbb{T}^3))}\leq C_{\delta,T}.
\end{aligned}
\end{equation}	
Similarly, due to $\gamma_{0}>4$, one has
\begin{equation}\nonumber
\begin{aligned}
&\var\int_{0}^{T}\int_{\mathbb{T}^3} \big( \nabla\rho \cdot \nabla (-\Delta)^{-1}\dive (\rho u)+\rho (-\Delta)^{-1}\dive\dive (\nabla \rho\otimes u) +\rho  (-\Delta)^{-1}\dive (|u|^{8}u)\big)dxdt\\
&\quad\leq \var \|\nabla\rho \|_{L^2(0,T;L^2)}\|\rho u\|_{L^2(0,T;L^2)}+C\var \|\rho \|_{L^{\infty}(0,T;L^{3}(\mathbb{T}^3))}  \|\nabla \rho \|_{L^2(0,T;L^2)}\|u\|_{L^2(0,T;L^6(\mathbb{T}^3))}\\
&\quad\quad+C\var  \|\rho \|_{L^{\infty}(0,T;L^{\frac{30}{13}}(\mathbb{T}^3))} \| |u|^{9}\|_{L^1(0,T;L^{\frac{10}{9}}(\mathbb{T}^3))}\leq C_{\delta,T}.
\end{aligned}
\end{equation}	
Since the other terms on the right-hand side of (\ref{rhoformular11}) can be controlled in a standard way as \cite{lions2,feireisl3}, we omit the detail.
	\end{proof}
	
	Inspired by \cite{l1}, we take the advantage of Lemma \ref{lemma31}-\ref{lemma34} and the De Giorgi iteration to have the upper and lower bounds of two densities.
	\begin{lemma}\label{lemma35}
		For any $\var\in(0,\frac{1}{4})$, $\delta\in(0,1)$ and given time $T>0$, let $(n,v,\rho,u)$ be any strong solution to the IVP $(\ref{twoapp})$ for $t\in (0,T]$. Then, under the assumptions of Theorem \ref{theorem11}, it holds
		\begin{equation}\label{upper}
			\begin{aligned}
				&0<\frac{1}{C_{\varepsilon,T}}\leq n(x,t)\leq C_{\varepsilon,T},\quad\quad0<\frac{1}{C_{\varepsilon,T}}\leq \rho(x,t)\leq C_{\varepsilon,T},\quad (x,t)\in\mathbb{T}^3\times(0,T),
			\end{aligned}
		\end{equation}
		where $C_{\varepsilon,T}>1$ is a constant.
	\end{lemma}

	\begin{proof}
		First, it holds by (\ref{energyvar}), (\ref{BDvar}) and the Sobolev inequality that
		\begin{align}
		\sup_{(x,t)\in\in\mathbb{T}^3\times(0,T)} n(x,t) \le C \sup_{t\in[0,T]} \| \sqrt{n}(t)\|_{L^2}^{\frac{2}{7}}\| \nabla \sqrt{n}(t) \|_{L^4(\mathbb{T}^3)}^{\frac{12}{7}} \le C_{\varepsilon,T}.  \label{33}
		\end{align}

		Next, the lower bound of $n$ can be shown by a De Giorgi-type procedure. It is easy to verify that $\vartheta:=n^{-\frac{1}{2}}$ satisfies
		\begin{align}
		&2\vartheta_{t}+2v\cdot\nabla\vartheta-\vartheta\dive v+\var\vartheta^{27}+2\var\vartheta^{-1}|\nabla \vartheta|^{2}+2\var\vartheta^{-5}|\nabla \vartheta|^4 =\var \Delta \vartheta+\var \dive(\vartheta^{-4}|\vartheta|^{2}\nabla\vartheta).    \label{29}
		\end{align}
		Multiplying $(\ref{29})$ by $(\vartheta-k)_{+}$ with $k \geq\|n_{0, \delta}^{-\frac{1}{2}}\|_{L^{\infty}(\mathbb{T}^3)}$ and integrating the resulted equation over $\mathbb{T}^3\times(0,t)$, we have
		\begin{equation}\label{3.32}
		\begin{aligned}
		&\sup_{t\in[0,T]}\int_{\mathbb{T}^3} |(\vartheta-k)_{+}|^{2}dx+\var\int_{0}^{T}\int_{\mathbb{T}^3} (|\nabla (\vartheta-k)_{+}|^{2}+ n^2 |\nabla (\vartheta-k)_{+}|^{4})dx \\\
		&\quad\leq3 \int_{0}^{T}\int_{\mathbb{T}^3} (\vartheta-k)_{+} |v| |\nabla\vartheta |dxdt+\int_{0}^{T}\int_{\mathbb{T}^3} \vartheta |v| |\nabla(\vartheta-k)_{+}|dxdt\\
		&\quad\leq 4\int_{0}^{T}\int_{\mathbb{T}^3} n^{-\frac{1}{2}} |v| |\nabla(\vartheta-k)_{+}|dxdt\\
		&\quad\leq \frac{\var}{2} \int_{0}^{T}\int_{\mathbb{T}^3}n^2 |\nabla (\vartheta-k)_{+}|^{4}dxdt+\frac{C}{\var}\int_{0}^{T}\int_{\mathbb{T}^3} n^{-\frac{4}{3}}|v|^{\frac{4}{3}} \mathbb{I}_{\vartheta>k}dxdt,
                \end{aligned}
		\end{equation}
		where one has used $(\vartheta-k)_{+}|_{t=0}=0$ and $(\vartheta-k)_{+}\leq \vartheta$. It follows from (\ref{energyvar}) that
		\begin{equation}\nonumber
		\begin{aligned}
		&\int_{0}^{T}\int_{\mathbb{T}^3} n^{-\frac{4}{3}}|v|^{\frac{4}{3}} \mathbb{I}_{\vartheta>k}dxdt\\
		&\quad\leq \big( \int_{0}^{T}\int_{\mathbb{T}^3} n|v|^{5}dxdt\big)^{\frac{4}{15}} \big( \int_{0}^{T}\int_{\mathbb{T}^3} n^{-\frac{24}{11}}\mathbb{I}_{\vartheta>k}dx\big)^{\frac{11}{15}} \\
		&\quad\leq T^{\frac{4}{5}} \big( \int_{0}^{T}\int_{\mathbb{T}^3} n|v|^{5}dxdt\big)^{\frac{4}{15}}  \big( \int_{0}^{T}\int_{\mathbb{T}^3} n^{-25}dx\big)^{\frac{8}{125}}  T^{\frac{251}{375}} |\vartheta>k|^{\frac{251}{375}}\leq C_{\var,T}  |\vartheta>k|^{\frac{251}{375}}.
		\end{aligned}
		\end{equation}
	      This together with $(\ref{3.32})$ yields
		\begin{align}
		&\sup_{t\in[0,T]}\int_{\mathbb{T}^3} (\vartheta-k)^{2}_{+}dx+\var\int_{0}^{T}\int_{\mathbb{T}^3} (|\nabla (\vartheta-k)_{+}|^{2}+n^2 |\nabla (\vartheta-k)_{+}|^{4})dxdt \leq C_{\var, T} |\vartheta>k|^{\frac{251}{375}}.\label{30}
		\end{align}
		On the other hand, using the Gagliardo-Nirenberg inequality, we deduce for any $h>k$ that
		\begin{equation}\label{31}
		\begin{aligned}
		&|\vartheta>h|\leq (h-k)^{-\frac{10}{3}} \|(\vartheta-k)_{+}\|_{L^{\frac{10}{3}}(0,T;L^{\frac{10}{3}})}^{\frac{10}{3}}\\
		&\quad\quad\quad\leq C (h-k)^{-\frac{10}{3}} \big(\sup_{t\in[0,T]}\int_{\mathbb{T}^3} |(\vartheta-k)_{+}|^2dx\big)^{\frac{4}{3}} \big(\int_{0}^{T} \int_{\mathbb{T}^3}|\nabla (\vartheta-k)_{+}|^{2}dxdt\big)^{\frac{1}{3}}.
		\end{aligned}
		\end{equation}
		Combining (\ref{30})-(\ref{31}) together, we get
		\begin{equation}\nonumber
		\begin{aligned}
		&|\vartheta>h|\leq C_{\var}(h-k)^{-\frac{10}{3}}|\vartheta>k|^{\frac{251}{225}},\quad\quad h>k.
		\end{aligned}
		\end{equation}
		Thus, by virtue of the De Giorgi lemma (see, e.g. \cite[Lemma 4.1.1]{wuz1}), there is a constant $C_{\var}>0$ such that
		\begin{align}
		\vartheta(x,t)\le C_{\var, T}, \quad\quad (x,t)\in\mathbb{T}^3\times(0,T),\nonumber
		\end{align}
         the lower bound of $n$ follows.

Furthermore, we turn to show the upper bound of $\rho$. By the Sobolev inequality, one obtains after multiplying $(\ref{twoapp})_{3}$ by $p\rho^{p-1}$ for $p\in[2,\infty)$ that
\begin{equation}\nonumber
\begin{aligned}
&\frac{d}{dt}\int_{\mathbb{T}^3}\rho^{p}dx+\var p (p-1)\int_{\mathbb{T}^3} \rho^{p-2}|\nabla \rho|^2dx\\
&\leq p(p-1)\|u\|_{L^{10}} \|\rho^{\frac{p}{2}-1}\nabla \rho\|_{L^2} \|\rho^{\frac{p}{2}}\|_{L^{\frac{5}{2}}}\\
&\leq p(p-1)\|u\|_{L^{10}}  \|\rho^{\frac{p}{2}-1}\nabla \rho\|_{L^2} (\|\rho^{\frac{p}{2}}\|_{L^{2}}+\|\rho^{\frac{p}{2}}\|_{L^{2}}^{\frac{7}{10}}\|\nabla \rho^{\frac{p}{2}}\|_{L^2}^{\frac{3}{10}})\\
&\leq \frac{\var p (p-1)}{2}\int_{\mathbb{T}^3} \rho^{p-2}|\nabla \rho|^2dx+C p(1+\|u\|_{L^{10}}^{10})\int_{\mathbb{T}^3}\rho^{p}dx,
\end{aligned}
\end{equation}
from which and (\ref{energyvar}) we deduce
\begin{equation}\label{Cprho}
\begin{aligned}
&\sup_{t\in[0,T]}\int_{\mathbb{T}^3}\rho^{p}dx\leq C_{p},\quad\quad p\in[2,\infty),
\end{aligned}
\end{equation}
where $C_{p}$ is a constant dependent of $\var$, $T$ and $p$. Thence we take the $L^2(\mathbb{T}^3)$ inner product of $(\ref{twoapp})_{3}$ with $(\rho-\ell)_{+}$ for $\ell \ge \| \rho_{0,\delta} \|_{L^{\infty}}$ to have
	\begin{equation}\label{34}
		\begin{aligned}
			&\sup_{t\in[0,T]}\int_{\mathbb{T}^3} |(\rho-\ell)_{+}|^{2}dx+\var\int_{0}^{T}\int_{\mathbb{T}^3} |\nabla (\rho-\ell)_{+}|^{2}dxdt \\
			&\quad\le \int_{0}^{T}\int_{\mathbb{T}^3} \rho u \cdot \nabla (\rho-\ell)_{+}dxdt   \\
			&\quad\le \frac{1}{\var}\int_{0}^{T}\int_{\mathbb{T}^3} \mathbb{I}_{\rho>\ell} \rho^{2}|u|^{2}dxdt+\frac{\var}{4}\int_{\mathbb{T}^3} |\nabla (\rho-\ell)_{+}|^{2}dxdt.
			\end{aligned}
		\end{equation}
		Due to (\ref{energyvar}) and $(\ref{Cprho})$, the first term on the right-hand side of $(\ref{34})$ can be estimated by
	\begin{equation}
		\begin{aligned}
			\int_{0}^{T}\int_{\mathbb{T}^3}\mathbb{I}_{\rho>\ell}\rho^{2}|u|^{2}dxdt &\le(\int_{0}^{T}\int_{\mathbb{T}^3} \rho^{50}dxdt)^{\frac{1}{25}}(\int_{0}^{T}\int_{\mathbb{T}^3}\mathbb{I}_{\rho>\ell}dxdt )^{\frac{19}{25}}(\int_{0}^{T}\int_{\mathbb{T}^3}  |u|^{10}dxdt)^{\frac{1}{5}}   \notag \\  
			& \le C_{\var,T}|\rho>\ell|^{\frac{19}{25}}. \notag 
			\end{aligned}
		\end{equation}
This implies for any $m>\ell$ that
	\begin{equation}\nonumber
		\begin{aligned}
			|\rho>m|&\leq (m-\ell)^{-\frac{10}{3}} \|(\rho-\ell)_{+}\|_{L^{\frac{10}{3}}(0,T;L^{\frac{10}{3}})}^{\frac{10}{3}}\\
		&\leq T^{\frac{2}{3}}(m-\ell)^{-\frac{10}{3}} \big(\sup_{t\in[0,T]}\int_{\mathbb{T}^3} (\rho-\ell)^{2}_{+}dx\big)^{\frac{4}{3}} \big(\int_{0}^{T} \int_{\mathbb{T}^3}|\nabla (\rho-\ell)_{+}|^{2}dxdt\big)^{\frac{1}{3}}\\
		&\leq C_{\var,T}(m-\ell)^{-\frac{10}{3}} |\rho>\ell|^{\frac{19}{15}}.
			\end{aligned}
		\end{equation}
		Thus, the De Giorgi lemma implies the upper bound of $\rho$.

		Finally, it is easy to verify that $\rho^{-1}$ satisfies
		\begin{align}
			(\rho^{-1})_{t}-\dive(\rho^{-1}u)+2u\cdot\nabla \rho^{-1}+2\var \rho|\nabla \rho^{-1}|^{2}=\var \Delta  \rho^{-1}.  \label{38}
		\end{align}
		Note that the negative term $2\var\rho|\nabla \rho^{-1}|^2$ in (\ref{38}) don't influence the energy estimates, and therefore we are able to estimate the $L^{\infty}(0,T;L^{p}(\mathbb{T}^3))$ of $\rho^{-1}$ for any $p\in[2,\infty)$ and employ a similar De Giorgi-type procedure to derive the lower bound of $\rho$. The proof of Lemma \ref{lemma35} is completed.
	
	\end{proof}

	\begin{lemma}\label{lemma36}
		For any $\var\in(0,\frac{1}{4})$, $\delta\in(0,1)$ and given time $T>0$, let $(n,v,\rho,u)$ be any strong solution to the IVP $(\ref{twoapp})$ for $t\in (0,T]$. Then, under the assumptions of Theorem \ref{theorem11}, it holds
		\begin{equation}\label{high}
			\begin{aligned}
				&\sup_{t\in[0,T]}\|(n,v,\rho, u)(t)\|_{H^2}+\int_{0}^{T}\big(\|(n,v,\rho,u)(t)\|_{H^3}^2+\|(n_{t},v_{t},\rho_{t},u_{t})(t)\|_{H^1}^2 \big)dt\leq C_{\varepsilon,T},
			\end{aligned}
		\end{equation}
		where $C_{\varepsilon,T}>0$ is a constant.
	\end{lemma}
	
	\begin{proof}
		First, it is easy to verify that
		\begin{align}
		2(\sqrt{n})_{t}-\var \dive((1+|\nabla \sqrt{n}|^{2})\nabla \sqrt{n})=-\dive(\sqrt{n} v+\nabla W)-\frac{1}{|\mathbb{T}^3|}\int_{\mathbb{T}^3} (v \cdot \nabla \sqrt{n}-\var \sqrt{n}^{-25})dx,   \label{43}
		\end{align}
		where $W(\cdot,t)$ for $t>0$ is the unique solution of the following elliptic problem:
		\begin{equation}
		\Delta W=v \cdot \nabla \sqrt{n}-\var \sqrt{n}^{-25}-\frac{1}{|\mathbb{T}^3|}\int_{\mathbb{T}^3} (v \nabla \sqrt{n}-\var \sqrt{n}^{-25})dx, \  \   \ \int_{\mathbb{T}^3} W dx=0. \label{44}
		\end{equation}
	    Applying \eqref{energyvar}, \eqref{BDvar},  $(\ref{upper})$ and $L^{p}$-estimates of the elliptic equation $(\ref{44})$(cf.\cite{wuz1}) that
		\begin{equation}\label{45}
			\begin{aligned}
		\|\nabla W (t)\|_{L^{p}} &\le \|\nabla^{2} W(t) \|_{L^{\frac{3p}{3+p}}(\mathbb{T}^3)}  \\
		&\le C_{p}\|v(t)\|_{L^{p}} \|\nabla \sqrt{n}(t)\|^{\frac{1}{3}}_{L^{2}}\|\nabla \sqrt{n}(t)\|^{\frac{2}{3}}_{L^{4}}+C_{p}   \le C_{p} \|v(t)\|_{L^{p}}, \quad t>0,\quad  p \in(1, \infty).  
		\end{aligned}
		\end{equation}
		Setting 
		$$
		W^{*}:=\sqrt{n}+\frac{1}{2|\mathbb{T}^3|}\int_{0}^{T}\int_{\mathbb{T}^3} (v\cdot\nabla W^{*}-\var \sqrt{n}^{-25})dxdt,$$
		 we rewrite $(\ref{43})$ by
		\begin{align}
		2W^{*}_{t}-\var \dive(|\nabla W^{*}|^{2}|\nabla W^{*}|)=\dive (\var \nabla W{*} -\sqrt{n}v-\nabla W ). \label{3.44}
		\end{align}
		One derives from \eqref{energyvar}, \eqref{BDvar},  $(\ref{upper})$, $(\ref{45})$ and $L^{p}$-estimates of the parabolic equation $(\ref{3.44})$ (cf.\cite{ac1}) that
		\begin{equation}
			\begin{aligned}
		\int_{0}^{T} \| \nabla \sqrt{n}(t) \|^{3p}_{L^{3p}}dt&=\int_{0}^{T} \| \nabla W^{*}(t) \|^{3p}_{L^{3p}}dt  \notag \\
		&\le C_{p}+C_{p}(\int_{0}^{T} \|v(t)\|^{p}_{L^p}dt)^{2}+\frac{1}{2}\int_{0}^{T} \| \nabla\sqrt{n}(t) \|^{3p}_{L^{3p}}dt,  \ \  \  p\in (1, \infty),  \notag
		\end{aligned}
		\end{equation}
	which implies 
	\begin{equation}
		\int_{0}^{T} \| \nabla \sqrt{n}(t)\|^{3p}_{L^{3p}}dt \le C_{p}+C_{p}(\int_{0}^{T} \|v(t)\|^{p}_{L^p}dt)^{2}, \ \ \ p\in(1, \infty).  \label{46}
	\end{equation}

		Next, we show the high-order estimates of $v$. Note that the equation $(\ref{twoapp})_{2}$ can be rewrite as
		\begin{equation}
		\begin{aligned}
		&v_{t}-(\frac{\eta}{2}+\sqrt{\var})\Delta v-\frac{\eta}{2}\nabla \dive v+\var|v|^3v+\kappa v\\
		&\quad=-v \cdot\nabla v-n^{-1} \nabla n +\kappa u-\var n^{-13}v+\var n^{-\frac{1}{2}}|\nabla\sqrt{n}|^2\nabla\sqrt{n} \cdot \nabla v.   \label{47}
		\end{aligned}
		\end{equation}
		Multiplying $(\ref{47})$ by $-2\Delta v$ and integrating the resulted equation over $\mathbb{T}^3$, we get by \eqref{energyvar} and \eqref{upper} that
		\begin{equation} \label{49}
			\begin{aligned}
		&\frac{d}{dt}\int_{\mathbb{T}^3} |\nabla v|^{2} dx+\int_{\mathbb{T}^3} ((\eta+2\sqrt{\var})|\Delta v|^2+\eta |\nabla \dive v|^2+2\var   |v|^3|\nabla v|^2+6\var |v|^3|\nabla|v\|^2+2\kappa |\nabla v|^2)dx\\
		&\le \|\Delta v\|_{L^2}(\|v\|_{L^{5}}\|\nabla v\|^{\frac{2}{5}}_{L^2}\|\Delta v\|^{\frac{3}{5}}_{L^2}+\|\nabla n\|_{L^2} +\|u\|_{L^2}+\|v\|_{L^2}  +\| |\nabla\sqrt{n}|^3 \|_{L^{5}}\|\nabla v\|^{\frac{2}{5}}_{L^{2}}\|\Delta   v\|^{\frac{3}{5}}_{L^2})   \\
		&\le \frac{\eta}{2}\|\Delta v\|^2_{L^2}+C(\|v\|_{L^{5}}^{5}+\| |\nabla\sqrt{n}|^3 \|_{L^{5}}^{5})\|\nabla v\|^2_{L^2}+C_{\var, T}, 
		\end{aligned}
	\end{equation}
		which together with $(\ref{46})$, $(\ref{49})$, the Gr${\rm\ddot{o}}$nwall inequality and the Sobolev inequality gives rise to
		\begin{equation}  \label{3.49}
		\begin{aligned}
		&\sup_{t\in[0,T]} \| \nabla v(t)\|_{L^2}+\|\nabla^2v\|_{L^2(0,T;L^2)}+\|v\|_{L^{10}(0,T;L^{10})}+\|\nabla v\|_{L^{\frac{10}{3}}(0,T;L^{\frac{10}{3}})}\\
		&\quad\quad+\|\nabla\sqrt{n}\|_{L^{30}(0,T;L^{30})}\le C_{\var, T}.  
	 \end{aligned}
	\end{equation}
		The Sobolev inequality (cf. \cite{la1}), \eqref{46} and $(\ref{3.49})$ thus imply that
		\begin{align}
		\| \nabla v \|_{L^{5}(0,T;L^{5})}  +\| v \|_{L^{q}(0,T;L^{q})} +\| \nabla\sqrt{n} \|_{L^{3q}(0,T;L^{3q})} \le C_{\var, T},\quad q\in[2,\infty), \label{3.51}
		\end{align}
	   and therefore it holds by \eqref{3.51}, $L^{p}$ estimates for the parabolic equations \eqref{3.44} and $(\ref{47})$ and the Sobolev inequality (cf. \cite{la1}) that
		\begin{equation}\label{3.53}
		\begin{aligned}
		&\|v\|_{L^{\infty}(0,T;L^{\infty})}+\| (\nabla^2 \sqrt{n},\nabla^2v)\|_{L^{4}(0,T;L^{4})}  \\
		&\quad\le C_{\var, T}(1+\|v\|_{L^{20}(0,T;L^{20})} \|\nabla v\|_{L^{5}(0,T;L^5)}+\|\nabla \sqrt{n}\|_{L^{4}(0,T;L^{4})}\\
		&\quad\quad +\|(v,u)\|_{L^{4}(0,T;L^{4})}+ \|\nabla\sqrt{n}|^3\|_{L^{4}(0,T;L^{4})}) \le C_{\var, T}
		\end{aligned}
	\end{equation}
	With the help of \eqref{energyvar}, \eqref{BDvar}, \eqref{upper} and \eqref{3.49}-\eqref{3.53}, we employ $L^{p}$ estimates for the parabolic equations \eqref{3.44} and $(\ref{47})$ again to gain
	 		\begin{equation}\nonumber
		\begin{aligned}
		&\sup_{t\in[0,T]} \|(n,v)(t)\|_{H^2}+\|(n,v)\|_{L^2(0,T;H^3)}+\|(n_{t},v_{t})\|_{L^2(0,T;H^1)}\\
		&\quad\leq C_{\var,T}(1+\|(v,u)\|_{L^2(0,T;H^1)}+\|\nabla\sqrt{n}\|_{L^2(0,T;H^1)}+\| |\nabla\sqrt{n}|^3\nabla v\|_{L^{2}(0,T;H^1)} )\leq C_{\var,T},
			\end{aligned}
	\end{equation}
	where we have used the fact
	 		\begin{equation}\nonumber
		\begin{aligned}
		&\| |\nabla\sqrt{n}|^3\nabla v\|_{L^{2}(0,T;H^1)} \\
		&\quad\leq \| |\nabla\sqrt{n}|^3\nabla v\|_{L^{2}(0,T;L^2(\mathbb{T}^3))}+\| |\nabla\sqrt{n}|^3\nabla^2 v\|_{L^{2}(0,T;L^2(\mathbb{T}^3))}+\|\nabla\sqrt{n}|^2\nabla^2\sqrt{n}:\nabla v\|_{L^{2}(0,T;L^2(\mathbb{T}^3))} )\\
		&\quad\leq \| |\nabla\sqrt{n}|^3\|_{L^4(0,T;L^4(\mathbb{T}^3))}(\|\nabla v\|_{L^{4}(0,T;L^{4})}+ \|\nabla^2 v\|_{L^{4}(0,T;L^{4})} )\\
		&\quad\quad+\| |\nabla\sqrt{n}|^2\|_{L^{20}(0,T;L^{20}(\mathbb{T}^3))}\|\nabla^2\sqrt{n}\|_{L^{4}(0,T;L^4(\mathbb{T}^3))}\|\nabla v\|_{L^5(0,T;L^5(\mathbb{T}^3))}\leq C_{\var,T}.
			\end{aligned}
	\end{equation}
By similar computations, one can show the expected high-order estimates of $(\rho,u)$. The proof of Lemma \ref{lemma36} is completed.
	\end{proof}

	\section{Vanishing artificial viscosities}\label{section4}
	
	In this section,  for $\delta\in(0,1)$ and $\gamma_{0}>\max\{\gamma+4\}$, we turn to consider the following IVP:
	\begin{equation}\label{twodelta}
		\left\{
		\begin{aligned}
			&n_{t}+\dive (nv)=0,\\
			&(nv)_{t}+\dive (nv\otimes v)+\nabla n= -\kappa n(v-u)+\eta \dive (n\mathbb{D} (v)),\\
			&\rho_{t}+\dive(\rho u)=0, \\
			&(\rho u )_{t}+\dive(\rho  u\otimes u)+\nabla(A\rho^{\gamma}+\delta \rho^{\gamma_{0}})=\kappa n(v-u)+\mu\Delta u+(\mu+\lambda)\nabla\dive u,\quad x\in\mathbb{T}^{3},\quad t>0,\\
		&(n,v,\rho, u)(x,0)=(n_{0,\delta},v_{0,\delta},\rho_{0,\delta},u_{0,\delta})(x),\quad x\in\mathbb{T}^{3}.
		\end{aligned}
		\right.
	\end{equation}
	where $(n_{0,\delta},v_{0,\delta},\rho_{0,\delta},u_{0,\delta})$ is given by \eqref{dapp}.
	
	We have the global existence of weak solutions to the IVP \eqref{twodelta} below.
	\begin{prop}\label{weakdelta}
		Let $\delta\in(0,1)$, and the assumptions of Theorem \ref{theorem11} hold. Then there exists a global weak solution $(n, nv,\rho,\rho u)$ to the IVP $(\ref{twodelta})$ with $n,\rho\geq0$ satisfying the following properties for any $T>0$$:$
		\begin{itemize}
			\item 
			The conservation laws of momentum and mass hold:
			\begin{align}
				&\int_{\mathbb{T}^3} ndx=\int_{\mathbb{T}^3}n_{0,\delta}dx,\quad\quad~ \int_{\mathbb{T}^3} \rho dx=\int_{\mathbb{T}^3} \rho_{0,\delta}dx,\quad t\in(0,T), \label{massdelta}\\
				&\int_{\mathbb{T}^3} (n v+\rho u )dx=\int_{\mathbb{T}^3}(n_{0,\delta}v_{0,\delta}+\rho_{0,\delta}u_{0,\delta})dx,\quad \quad\quad t\in(0,T).\label{momentumdelta}
			\end{align}
			
			\item{}The energy inequality holds: 
			\begin{equation}\label{energydelta}
				\begin{aligned}
					&\underset{t\in[0, T]}{{\rm{ess~sup}}}~\int_{\mathbb{T}^3}(\frac{1}{2}n|v|^2+n\log{n}-n+1+\frac{1}{2}\rho |u|^2+\frac{A\rho^{\gamma}}{\gamma-1}+\frac{\delta \rho^{\gamma_{0}}}{\gamma_{0}-1})dx\\
					&\quad\quad+\int_{0}^{t}\int_{\mathbb{T}^3} (\kappa n|v-u|^2+\mu |\nabla u|^2+(\mu+\lambda)(\dive u)^2 )dxd\tau\leq E_{0,\delta},
				\end{aligned}
			\end{equation}
			where $E_{0,\delta}$ is given by $(\ref{E0delta})$.
			
			\item Bresch-Desjardins type entropy inequality holds:
			
			There exists a constant $C>0$ uniformly in $\delta$ and $T$ such that
				\begin{align}
				&\sup_{t\in[0,T]}\int_{\mathbb{T}^3} |\nabla \sqrt{n}|^2dx+\int_{0}^{T}\int_{\mathbb{T}^3} |\nabla \sqrt{n}|^2dxdt  \le C,    \label{BDuniformintime}
			\end{align}

			\item{ }The Mellet-Vasseur type estimate follows: 
			
			There exists a constant $C_{T}>0$ independent of $\delta$ such that
			\begin{equation}\label{melletdelta}
				\begin{aligned}
					&\underset{t\in[0, T]}{{\rm{ess~sup}}}~\int_{\mathbb{T}^3} n(1+|v|^2)\log{(1+|v|^2)}dx\leq C_{T}.
				\end{aligned}
			\end{equation}
		\end{itemize}
	\end{prop}

	\vspace{2ex}

	\textbf{Proof of Proposition \ref{weakdelta}.} For $\var\in (0,\frac{1}{4})$, $\delta\in(0,1)$ and $\gamma_{0}>\max\{\gamma+4\}$, let $(n_{\var},v_{\var},\rho_{\var},u_{\var})$ be the global strong solution to the approximate problem (\ref{twoapp})-(\ref{dapp22}) given by Proposition \ref{appwell}. For any given time $T>0$, it follows from the a-priori estimates established in Lemmas \ref{lemma31}-\ref{lemma34} that there is a limit $(n,m,\rho,u)$ such that as $\varepsilon\rightarrow0$, we have
	\begin{equation}\label{limitvar1}
		\left\{
		\begin{aligned}
			&n_{\var}\overset{\ast}{\rightharpoonup} n\quad\quad~\text{in} \quad L^{\infty}(0,T; L^1(\mathbb{T}^3)),\\
			&n_{\var} v_{\var}\overset{\ast}{\rightharpoonup}  m\quad~\text{in}\quad L^{\infty}(0,T;L^1(\mathbb{T}^3)),\\
			&\rho_{\var}\rightharpoonup \rho\quad\quad~~ \text{in}\quad L^{\infty}(0,T;L^{\gamma_{0}}(\mathbb{T}^3)),\\
			&u_{\var} \rightharpoonup u\quad\quad~~\text{in}\quad L^2(0,T;H^1(\mathbb{T}^3)).
		\end{aligned}
		\right.
	\end{equation}
	One deduces by (\ref{energyvar}) and (\ref{BDvar}) that
	\begin{equation}\nonumber
		\begin{aligned}
			&\sup_{\var\in(0,\frac{1}{4})}\|\nabla n_{\var}\|_{L^{\infty}(0,T;L^{\frac{3}{2}})}\leq 2\sup_{\var\in(0,\frac{1}{4})} \big{(} \|\sqrt{n_{\var}}\|_{L^{\infty}(0,T;L^{6})}\|\nabla \sqrt{n_{\var}}\|_{L^{\infty}(0,T;L^2)} \big{)}<\infty,
		\end{aligned}
	\end{equation}
	and
	\begin{equation}\nonumber
		\begin{aligned}
			&\sup_{\var\in(0,\frac{1}{4})}\| (n_{\var})_{t}\|_{L^{2}(0,T;W^{-1,1})}\\
			&\quad\leq C\sup_{\var\in(0,\frac{1}{4})} \big{(} \|n_{\var}v_{\var}\|_{L^{2}(0,T;L^{1})}+\var\|\sqrt{n_{\var}}\Delta\sqrt{n_{\var}}\|_{L^{2}(0,T;L^1)}+\var\| \sqrt{n_{\var}}|\nabla\sqrt{n_{\var}}|^3\|_{L^{2}(0,T;L^1)}\\
			&\quad\quad+\var\|n_{\var}^{-12}\|_{L^2(0,T;L^2)}\big)\\
			&\quad\leq C\sup_{\var\in(0,\frac{1}{4})} \big{(} \|\sqrt{n_{\var}}\|_{L^{\infty}(0,T;L^{2})}\|\sqrt{n_{\var}}v_{\var}\|_{L^{\infty}(0,T;L^{2})}+\var\|\sqrt{n_{\var}}\|_{L^{\infty}(0,T;L^2)}\|\Delta\sqrt{n_{\var}}\|_{L^{2}(0,T;L^1)}\\
			&\quad\quad+\var\| \sqrt{n_{\var}}\|_{L^{\infty}(0,T;H^1)}\|\nabla \sqrt{n_{\var}}\|_{L^{6}(0,T;L^6)}+\var\|n_{\var}^{-25}\|_{L^1(0,T;L^1)}^{\frac{12}{25}}\big)\leq C(1+\var^{\frac{1}{25}})<\infty.
		\end{aligned}
	\end{equation}
	where in the last inequality we have used the fact
	\begin{equation}\nonumber
		\begin{aligned}
			&\var^{\frac{4}{3}}\int_{0}^{T}\|\nabla \sqrt{n_{\var}}\|_{L^{6}}^{6}dt\leq \int_{0}^{T}\|\nabla\sqrt{n_{\var}}\|_{L^2}^{\frac{2}{3}}\big{(}\var \|\nabla\sqrt{n_{\var}}\|_{L^4}^4\big)^{\frac{1}{3}}\big(\var\|\nabla\sqrt{n_{\var}}\|_{L^{12}}^4\big)dt\\
			&\quad\quad\quad\quad\quad\quad\quad\quad~~\leq C\var\int_{0}^{T}\| |\nabla \sqrt{n_{\var}}| |\nabla^2\sqrt{n_{\var}}|\|_{L^2}^2dt\leq C.
		\end{aligned}
	\end{equation}
	Thus, one employs the Aubin-Lions lemma and $W^{1,\frac{3}{2}}(\mathbb{T}^3)\hookrightarrow\hookrightarrow L^p(\mathbb{T}^3)$ $(1\leq p<3)$ to derive (up to a sequence) that
	\begin{equation}\label{limitvar2}
	\left\{
		\begin{aligned}
			&n_{\var}\rightarrow n\quad\quad\quad~~\text{in}\quad C([0,T];L^{p}(\mathbb{T}^3)),\quad 1\leq p<3,\\
			&\sqrt{n_{\var}}\overset{\ast}{\rightharpoonup} \sqrt{n}\quad\quad\text{in} \quad L^{\infty}(0,T; H^1(\mathbb{T}^3)),\\
			&\sqrt{n_{\var}}\to \sqrt{n}\quad\quad\text{in}\quad C([0,T];L^{p}(\mathbb{T}^3)),\quad 1\leq p<6.
		\end{aligned}
		\right.
	\end{equation}
	Due to (\ref{energyvar}) and (\ref{BDvar}), it is easy to show
	\begin{equation}\nonumber
		\begin{aligned}
			&\sup_{\var\in(0,\frac{1}{4})}\big( \|\nabla(n_{\var} v_{\var})\|_{L^2(0,T;L^1)}+\|(n_{\var} v_{\var})_{t}\|_{L^1(0,T;H^{-s_{0}})}\big)<\infty,
		\end{aligned}
	\end{equation}
	for some suitably large constant $s_{0}>0$. Therefore, we apply Aubin-Lions lemma and Lemma C.1 in \cite{lions1} to gain
	\begin{equation}\label{limitvar4}
		\left\{
		\begin{aligned}
			&n_{\var}v_{\var}\rightarrow n v\quad\quad\text{in}\quad L^2(0,T;L^{p}(\mathbb{T}^3)),\quad 1\leq p<\frac{3}{2},\\
			&n_{\var}v_{\var}\rightarrow nv\quad\quad\text{in}\quad C([0,T];L_{\text{weak}}^{\frac{3}{2}}(\mathbb{T}^3)),
		\end{aligned}
		\right.
	\end{equation}
		where $C([0,T]; X_{weak})$ denote the space of continuous functions on $[0,T]$ with values in $X$ equipped with the weak topology, and $v$ is given by
	\begin{equation}\nonumber
		v:=
		\begin{cases}
			\frac{m}{n},\quad
			& \mbox{if $n>0,$ } \\
			0. \quad
			& \mbox{if $n=0.$}
		\end{cases}
	\end{equation}
	Similarly, one gets as $\var\rightarrow0$ that
	\begin{equation}\label{limitvar11}
		\left\{
		\begin{aligned}
			&\rho_{\var}\rightarrow \rho\quad\quad~~\quad\quad\text{in}\quad C([0,T];H^{-1}(\mathbb{T}^3)),\\
			&\rho_{\var}\rightarrow \rho\quad\quad~~\quad\quad\text{in}\quad C([0,T];L^{\gamma}_{weak}(\mathbb{T}^3)),\\
			&\rho_{\var}u_{\var}\rightharpoonup \rho u\quad\quad\quad\text{in}\quad C([0,T];L^{\frac{2\gamma}{\gamma+1}}_{weak}(\mathbb{T}^3)),\\
			&\rho_{\var}u_{\var}\rightarrow \rho u\quad\quad\quad\text{in}\quad C([0,T];H^{-1}(\mathbb{T}^3)),
		\end{aligned}
		\right.
	\end{equation}
 By (\ref{limitvar3})-(\ref{limitvar11}), it also holds
	\begin{equation}\label{limitvar112}
		\left\{
		\begin{aligned}
			&\rho_{\var}u_{\var}\otimes u_{\var}\rightarrow \rho u\otimes u~~\quad\quad\quad\text{in}\quad \mathcal{D}'(\mathbb{T}^3\times(0,T)),\\
			&n_{\var}v_{\var}-n_{\var}u_{\var}\rightarrow n v-nu\quad\quad\text{in}\quad \mathcal{D}'(\mathbb{T}^3\times(0,T)).
		\end{aligned}
		\right.
	\end{equation}
	In addition, as $\var \rightarrow 0$, it is easy to verify that
	\begin{equation}\label{limitvar5}
		\left\{
		\begin{aligned}
			&\var \sqrt{n_{\var}}\Delta \sqrt{n_{\var}}+\var\sqrt{n}\dive (|\nabla \sqrt{n_{\var}}|^2\nabla \sqrt{n_{\var}}) +\var n_{\var}^{-12}\rightarrow 0\quad~\text{in}\quad\mathcal{D}'(\mathbb{T}^3\times(0,T)),\\
			&\sqrt{\varepsilon} \dive (n_{\var}\nabla v_{\var})+\varepsilon\sqrt{n_{\var}} \dive (|\nabla \sqrt{n_{\var}}|^2\nabla \sqrt{n_{\var}} v_{\var}) \rightarrow 0\quad\quad~\text{in}\quad\mathcal{D}'(\mathbb{T}^3\times(0,T)),\\
			&\varepsilon \sqrt{n_{\var}}\Delta \sqrt{n_{\var}} -\var n_{\var}|v_{\var}|^3v_{\var} \rightarrow 0\quad\quad\quad\quad\quad\quad\quad\quad\quad\quad~~~~\text{in}\quad\mathcal{D}'(\mathbb{T}^3\times(0,T)),\\
			&\var\Delta\rho_{\var}\rightarrow 0\quad\quad\quad\quad\quad\quad\quad\quad\quad\quad\quad\quad\quad\quad\quad\quad\quad\quad\quad~~\text{in}\quad\mathcal{D}'(\mathbb{T}^3\times(0,T)),\\
			&-\varepsilon\nabla u_{\var}\cdot\nabla \rho_{\var}-\var |u_{\var}|^8u_{\var}\rightarrow 0\quad\quad\quad\quad\quad\quad\quad\quad\quad\quad~~~~\text{in}\quad\mathcal{D}'(\mathbb{T}^3\times(0,T)).
		\end{aligned}
		\right.
	\end{equation}

	Next, we gain the strong convergence of $\sqrt{n_{\var}}v_{\var}$. For any constant $L>0$, write
	\begin{equation}\label{me1}
		\begin{aligned}
			&\int_{0}^{T}\int_{\mathbb{T}^3}|\sqrt{n_{\var}}v_{\var}-\sqrt{n}v|^2dxdt\\
			&\quad\leq \int_{0}^{T}\int_{\mathbb{T}^3}|\sqrt{n_{\var}}v_{\var}\mathbb{I}_{\{|v_{\var}|\leq L\}}-\sqrt{n}v\mathbb{I}_{\{ |v|\leq L\}}|^2dxdt\\
			&\quad\quad+2\int_{0}^{T}\int_{\mathbb{T}^3}|\sqrt{n_{\var}}v_{\var}\mathbb{I}_{ \{|v_{\var}|\geq L\}}|^2dxdt+2\int_{0}^{T}\int_{\mathbb{T}^3}|\sqrt{n}v\mathbb{I}_{ \{ |v|\geq L\}} |^2dxdt.
		\end{aligned}
	\end{equation}
	We need to show that the first term on the right-hand side of (\ref{me1}) tends to $0$ as $\var\rightarrow0$. Let $\zeta>0$, and $Q_{\zeta}:=\{(x,t)\in\mathbb{T}^3\times(0,T)~|~n(x,t)\geq \zeta>0\}$. According to (\ref{limitvar2}) and the Egorov theorem, for any $\zeta'>0$, there is a sufficiently small constant $\var_{1}\in(0,\var_{0})$ and a set $\tilde{Q}_{\zeta'}\subset Q_{\zeta}$  such that for any $\var\in(0,\var_{1})$, it holds
	\begin{equation} \label{y1}
		\begin{aligned}
			&|Q_{\zeta}/\tilde{Q}_{\zeta'}|\leq \zeta',\quad \quad n(x,t)\geq \frac{\zeta}{2}>0,\quad\quad (x,t)\in \tilde{Q}_{\zeta'}.
		\end{aligned}
	\end{equation}
	Thus, there follows by (\ref{limitvar2}), (\ref{limitvar4}) and (\ref{y1}) that
	\begin{equation}\nonumber
		\begin{aligned}
			&\sqrt{n_{\var}}v_{\var}\rightarrow \sqrt{n}v\quad\quad\text{a.e. in}\quad \tilde{Q}_{\zeta'},
		\end{aligned}
	\end{equation}
	from which and the dominated convergence theorem we infer
	\begin{equation}\label{y2}
		\begin{aligned}
			\lim_{\var\rightarrow0}\int_{\tilde{Q}_{\zeta'}}|\sqrt{n_{\var}}v_{\var}\mathbb{I}_{ \{ |v_{\var}|\leq L \}}-\sqrt{n}v\mathbb{I}_{ \{ |v|\leq L \}}|^2dxdt=0.
		\end{aligned}
	\end{equation}
	Combining (\ref{y1})-(\ref{y2}) together, we have for any fixed $L,\zeta,\zeta'>0$ that
	\begin{equation}\nonumber
		\begin{aligned}
			&\int_{0}^{T}\int_{\mathbb{T}^3}|\sqrt{n_{\var}}v_{\var}\mathbb{I}_{ \{|v_{\var}|\leq L\}}-\sqrt{n}v\mathbb{I}_{ \{ |v|\leq L \}}|^2dxdt\\
			&\quad\leq \int_{\tilde{Q}_{\zeta'}}|\sqrt{n_{\var}}v_{\var}\mathbb{I}_{ \{ |v_{\var}|\leq L \}}-\sqrt{n}v\mathbb{I}_{ \{ |v|\leq L \}}|^2dxdt+\int_{Q_{\zeta}/\tilde{Q}_{\zeta'}} |\sqrt{n_{\var}}v_{\var}\mathbb{I}_{ \{ |v_{\var}|\leq L \}}-\sqrt{n}v\mathbb{I}_{ \{ |v|\leq L \}}|^2dxdt\\
			&\quad\quad+\int_{\mathbb{T}^3\times(0,T)/Q_{\zeta}}|(\sqrt{n_{\var}}-\sqrt{n})v_{\var}\mathbb{I}_{ \{ |v_{\var}|\leq L \}}+\sqrt{n}(v\mathbb{I}_{ \{ |v|\leq L \}}-v_{\var}\mathbb{I}_{ \{ |v_{\var}|\leq L \}})|^2dxdt\\
			&\quad\leq \int_{\tilde{Q}_{\zeta'}}|\sqrt{n_{\var}}v_{\var}\mathbb{I}_{ \{ |v_{\var}|\leq L \}}-\sqrt{n}v\mathbb{I}_{ \{ |v|\leq L \}}|^2dxdt+CL \int_{0}^{T}\int_{\mathbb{T}^3} |\sqrt{n_{\var}}-\sqrt{n}|dxdt+CL(\sqrt{\zeta'}+\sqrt{\zeta}),
		\end{aligned}
	\end{equation}
	which together with (\ref{limitvar3}), (\ref{y2}) and the arbitrariness of $\zeta',\zeta$ leads to
	\begin{equation}\label{me2}
		\begin{aligned}
			&\lim_{\var\rightarrow0} \int_{0}^{T}\int_{\mathbb{T}^3}|\sqrt{n_{\var}}v_{\var}\mathbb{I}_{ \{|v_{\var}|\leq L\}}-\sqrt{n}v\mathbb{I}_{ \{ |v|\leq L \}}|^2dxdt=0,\quad\quad L>0.
		\end{aligned}
	\end{equation}
	We estimate the last term on the right-hand side of (\ref{me1}). For a.e. $(x,t)\in \{(x,t)\in\mathbb{T}^3\times(0,T)~|~n(x,t)>0\}$, there follows by (\ref{limitvar2}), (\ref{limitvar4}) and the fact $n_{\var}>0$ that
	\begin{equation}\label{mee1}
		\begin{aligned}
			&n |v|^2\log{(1+|v|^2)}=\frac{|m|^2\log{(n^2+|m|^2)}-2|m|^2\log{n}}{n}\\
			&\quad\quad\quad\quad\quad\quad\quad~~=\lim_{\var\rightarrow0}\frac{|n_{\var}v_{\var}|^2\log{(n_{\var}^2+|n_{\var}v_{\var}|^2)}-2|n_{\var}v_{\var}|^2\log{n_{\var}}}{n_{\var}}\\
			&\quad\quad\quad\quad\quad\quad\quad~~=\lim_{\var\rightarrow 0}n_{\var}|v_{\var}|^2\log{(1+|v_{\var}|^2)}.
		\end{aligned}
	\end{equation}
	It clearly holds on $\{(x,t)\in\mathbb{T}^3\times(0,T)~|~n(x,t)=0\}$ that
	\begin{equation}\label{mee2}
		\begin{aligned}
			&n |v|^2\log{(1+|v|^2)}=0\leq n_{\var} |v_{\var}|^2\log{(1+|v_{\var}|^2)}.
		\end{aligned}
	\end{equation}
	By (\ref{mee1})-(\ref{mee2}) and the Fatou lemma, we prove for a.e. $t\in(0,T)$ that
	\begin{equation}\label{melletdelta122}
		\begin{aligned}
			&\int_{\mathbb{T}^3} n(1+|v|^2)\log{(1+|v|^2)}dx\leq\liminf_{\var\rightarrow 0} \int_{\mathbb{T}^3} n_{\var}(1+|v_{\var}|^2)\log{(1+|v_{\var}|^2)}dx\leq C_{T}.
		\end{aligned}
	\end{equation}
	Due to (\ref{melletdelta122}), the last term on the right-hand side of (\ref{me1}) can be controlled by
	\begin{equation}\label{me3}
		\begin{aligned}
			\int_{0}^{T}\int_{\mathbb{T}^3}|\sqrt{n}v\mathbb{I}_{|v|\geq L}|^2dxdt\leq\frac{\int_{0}^{T}\int_{\mathbb{T}^3} n|v|^2\log{( 1+|v|^2)}dx}{\log{(1+L^2)}}\leq \frac{C_{T}}{\log{(1+L^2)}}.
		\end{aligned}
	\end{equation}
	Similarly, one gets
	\begin{equation}\label{me4}
		\begin{aligned}
			\int_{0}^{T}\int_{\mathbb{T}^3}|\sqrt{n_{\var}}v_{\var}\mathbb{I}_{|v_{\var}|\geq L}|^2dxdt\leq \frac{C_{T}}{\log{(1+L^2)}}.
		\end{aligned}
	\end{equation}
	Since the constant $L>0$ can be arbitrarily large, we substitute the estimates (\ref{me3})-(\ref{me4}) into (\ref{me1}) and then make use of  (\ref{me2}) to obtain
	\begin{equation}\label{4.23}
		\begin{aligned}
			&\sqrt{n_{\var}}v_{\var}\rightarrow \sqrt{n}v\quad\quad\quad\quad\quad\text{in}\quad L^2(0,T;L^2(\mathbb{T}^3)),
		\end{aligned}
	\end{equation}
	which leads to
	\begin{equation}\label{limitvar57}
		\begin{aligned}
			&n_{\var}v_{\var}\otimes v_{\var}\rightarrow \sqrt{n}v\otimes \sqrt{n}v\quad\text{in}\quad L^1(0,T;L^1(\mathbb{T}^3)).
		\end{aligned}
	\end{equation}
	
	Then, we aim to prove that the limit $(n,nv,\rho,\rho u)$ satisfies the conservation laws of mass and momentum. 
	One can take advantage of $(\ref{limitvar11})_{2}$, $(\ref{limitvar5})$ and the fact $1\in L^{\frac{\gamma}{\gamma-1}}(\mathbb{T}^3)$ to deduce 
	\begin{equation}\label{ma33}
		\begin{aligned}
			\lim_{\var\rightarrow0}\int_{\mathbb{T}^3}\rho_{\var}dx=\int_{\mathbb{T}^3}\rho_{0, \delta} dx,\quad\quad t\in(0,T).
		\end{aligned}
	\end{equation}
	 Similarly, one can derive
	\begin{align}
		&\lim_{\var\rightarrow0}\int_{\mathbb{T}^3}n_{\var}dx=\int_{\mathbb{T}^3}n_{0, \delta} dx,\quad\quad\quad\quad\quad\quad\quad\quad\quad\quad\quad\quad\quad~ t\in(0,T),  \label{4.25} \\
		&\int_{\mathbb{T}^3} nv dx=\int_{\mathbb{T}^3}n_{0,\delta}v_{0,\delta}dx+\int_{0}^{t}\int_{\mathbb{T}^3}\kappa (nu-nv)dxd\tau,\quad\quad t\in(0,T),\label{ee1}\\
		&\int_{\mathbb{T}^3} \rho u dx=\int_{\mathbb{T}^3}\rho_{0,\delta}u_{0,\delta}dx+\int_{0}^{T}\int_{\mathbb{T}^3}\kappa (nv-nu)dxdt,\quad\quad t\in(0,T).\label{ee2}
	\end{align}
	By (\ref{limitvar57})-(\ref{ee2}, (\ref{momentumdelta}) follows.

	Then, we claim as $\var\rightarrow0$ that
	\begin{equation}\label{strongrhoae}
		\begin{aligned}
			&\rho_{\var}\rightarrow \rho\quad\quad\text{a.e. in}\quad\mathbb{T}^3\times(0,T), 
		\end{aligned}
	\end{equation}
	which together with (\ref{gamma1}) and the Egorov theorem implies
	\begin{equation}\label{strongPvar}
		\begin{aligned}
			&A\rho_{\var}^{\gamma}+\delta \rho_{\var}^{\gamma_{0}}\rightarrow A\rho^{\gamma}+\delta \rho^{\gamma_{0}}\quad\quad\text{in}\quad L^1(0,T;L^1(\mathbb{T}^3)).
		\end{aligned}
	\end{equation}
	By virtue of (\ref{limitvar1})-(\ref{strongPvar}), one can conclude that the limit $(n,nv,\rho,\rho u)$ indeed satisfies the equations (\ref{twodelta}) in the sense of distributions. 

	\vspace{2ex}

	To prove (\ref{strongrhoae}), we need the following property of the effect viscous flux, which can be shown by repeating same arguments as used in \cite{feireisl1,lions2}.
	\begin{lemma}\label{lemma43}
		There holds
		\begin{equation}\label{effect1}
			\begin{aligned}
				&\lim_{\var\rightarrow 0}\int_{0}^{T}\int_{\mathbb{T}^{3}}\big{(}A\rho_{\var}^{\gamma}+\delta\rho_{\var}^{\gamma_{0}}-(2\mu+\lambda)\dive u_{\var}\big{)}\rho_{\var}dxdt\\
				&~~=\int_{0}^{T}\int_{\mathbb{T}^{3}}\big{(}A\overline{\rho^{\gamma}}+\delta\overline{\rho^{\gamma_{0}}}-(2\mu+\lambda)\dive u\big{)}\rho dxdt.
			\end{aligned}
		\end{equation}
		where $\overline{f}$ denotes the weak limit of $f_{\varepsilon}$ as $\varepsilon\rightarrow0$.
	\end{lemma}
	
	\vspace{2ex}

	\textbf{Proof of the strong convergence of of the density $\rho_{\var}$.}
	We are going to show (\ref{strongrhoae}). First, from $u \in L^2(0,T;H^1(\mathbb{T}^3))$, $\rho \in L^2(0,T;L^2(\mathbb{T}^3))$ and the arguements of renormalized solutions (cf. \cite{lions1,diperna1}) $(\rho, u)$ satisfies
	\begin{align}
		(\rho \log \rho )_{t}+\dive(\rho\log \rho)+\rho\dive u=0 \  \  \ \text{in} \  \mathcal{D}( \mathbb{T}^{3} \times (0, T)). \label{63}
	\end{align}
	One deduces from $(\ref{63})$ that
	\begin{align}
		\int_{0}^{t}\int_{\mathbb{T}^3} \rho \dive u dxd\tau=\int_{\mathbb{T}^3} \rho_{0}\log \rho_{0} dx-\int_{\mathbb{T}^3} \rho\log \rho dx,  \  \ t \in (0, T). \label{64}
	\end{align}
	Next, multiplying $(\ref{twoapp})_3$ by $1+\log \rho_{\var}$, we have
	\begin{align}
		&\int_{0}^{t}\int_{\mathbb{T}^3} \rho_{\var} \dive u_{\var} dxd\tau
		 \le  \int_{\mathbb{T}^3} \rho_{0, \delta}\log \rho_{0, \delta} dx-\int_{\mathbb{T}^3} \rho_{\var}\log \rho_{\var} dx, \  t \in (0, T). \label{66}
	\end{align}
    It follows by $(\ref{effect1})$ and $(\ref{64})$-$(\ref{66})$ that
    \begin{align}
    	&\overline{\lim_{\var \to 0}} \int_{\mathbb{T}^3} (\rho_{\var}\log \rho_{\var}-\rho\log \rho)dx  \notag \\
    	&\le \overline{\lim_{\var \to 0}} \int_{0}^{T}\int_{\mathbb{T}^3} (\rho\dive u-\rho_{\var}\dive u_{\var})dxd\tau  \notag \\
    	&=\overline{\lim_{\var \to 0}} \int_{0}^{T}\int_{\mathbb{T}^3} (\rho(A\rho^{\gamma}+\delta\rho^{\gamma_{0}})-\rho_{\var}(A\rho_{\var}^{\gamma}+\delta\rho_{\var}^{\gamma_{0}})dxd\tau  \le 0,   \notag
    \end{align}
    where we have used the monotonicity of $ A\rho^{\gamma}+\delta\rho^{\gamma_{0}}$. Since $\rho \to \rho\log\rho$ is strictly convex, we end up with $(\ref{strongrhoae})$.
	\textbf{Proof of the uniform-in-time basic energy inequality and Bresch-Desjardins estimate.} By $(\ref{twoapp})_{1}$, \eqref{limitvar2} and \eqref{4.23}, we show as $\var \to 0$   that
	\begin{equation}
		n_{t}+\dive (nv)=0. \label{two11}
	\end{equation}
Since it follows
\begin{equation}\label{4.100}
	\begin{aligned}
		&\int_{\mathbb{T}^3} (n_{\var}\log n_{\var}-n_{\var}+1)dx  \\
		&=\int_{0}^{t}\int_{\mathbb{T}^3} \log n_{\var} (\var \sqrt{n_{\var}}\Delta \sqrt{n_{\var}}+\var\sqrt{n_{\var}}\dive (|\nabla \sqrt{n_{\var}}|^2\nabla \sqrt{n_{\var}}) +\var n_{\var}^{-12})dxd\tau   \\
		&\quad+2\int_{0}^{t}\int_{\mathbb{T}^3} \sqrt{n_{\var}}\nabla \sqrt{n_{\var}}v_{\var} dxd\tau,\quad 0<t<T,
		\end{aligned}
\end{equation}
we have after taking the limit as $\var \to 0$ in \eqref{4.100} and employing $(\ref{limitvar2})$ and $(\ref{4.23})$ that
\begin{equation}
	\int_{\mathbb{T}^3} (n\log n-n+1)dx=2\int_{0}^{t}\int_{\mathbb{T}^3}  \sqrt{n}v\cdot\nabla \sqrt{n} dxd\tau,\quad 0<t<T. \label{4.28}
\end{equation}
	 Meanwhile, we make use of the lower semi-continuity of weak limits in $(\ref{energyloss})$ to get
	\begin{equation}\label{energyuniform}
	\begin{aligned}
		&\underset{t\in[0, T]}{{\rm{ess~sup}}}~\int_{\mathbb{T}^3}\big{(} \frac{1}{2}n|v|^{2} +\frac{1}{2}\rho |u|^2+\frac{A\rho^{\gamma}}{\gamma-1}+\frac{\delta \rho^{\gamma_{0}}}{\gamma_{0}-1}\big{)}dx\\
		&\quad\quad+\int_{0}^{t}\int_{\mathbb{T}^3}\big(\kappa n|v-u|^2+\mu |\nabla u|^2+(\mu+\lambda)(\dive u)^2\big)dxd\tau\\	
		&\quad\le \underset{t\in[0, T]}{{\rm{ess~sup}}}~\int_{\mathbb{T}^3}\big{(} \frac{1}{2}n_{\var}|v_{\var}|^{2} +\frac{1}{2}\rho_{\var} |u_{\var}|^2+\frac{A\rho_{\var}^{\gamma}}{\gamma-1}+\frac{\delta \rho_{\var}^{\gamma_{0}}}{\gamma_{0}-1}\big{)}dx\\
		&\quad\quad+\underset{\var \to 0}{{\rm{lim~inf}}}~\int_{0}^{t}\int_{\mathbb{T}^3}\big(\kappa n_{\var}|v_{\var}-u_{\var}|^2+\mu |\nabla u_{\var}|^2+(\mu+\lambda)(\dive u_{\var})^2\big)dxd\tau\\
		&\quad\quad+\underset{\var \to 0}{{\rm{lim~inf}}}~\var \int_{0}^{t}\int_{\mathbb{T}^3}  \big( n_{\var}|v_{\var}|^{5}+(1+|v_{\var}|^2)|\nabla \sqrt{n_{\var}}|^2(1+|\nabla \sqrt{n_{\var}}|^2)+ n_{\var}^{-25}\big)dxd\tau  \\
		&\quad\le E_{0, \delta}-2\int_{0}^{t}\int_{\mathbb{T}^3} \sqrt{n}v\cdot \nabla \sqrt{n} dxd\tau,
	\end{aligned}
\end{equation}
where $E_{0, \delta}$ is given by $(\ref{E0delta})$. Adding $(\ref{4.28})$ and $(\ref{energyuniform})$ together, we show  $(\ref{energydelta})$.	

    Similarly, we prove the uniform-in-time Bresch-Desjardins estimate. By $(\ref{3.4})$ and $(\ref{BDii})$, it holds as $\var \to 0$ that
\begin{equation}\label{BDiiuniform}
	\begin{aligned}
		&\underset{t\in[0, T]}{{\rm{ess~sup}}}~\int_{\mathbb{T}^3} (\frac{1}{2}n|v+c_{0}\nabla \log n|^2+\frac{1}{2}\rho |u|^2+\frac{A\rho^{\gamma}}{\gamma-1}+\frac{\delta \rho^{\gamma_{0}}}{\gamma_{0}-1})dx  \\
		&\quad+\int_{0}^{t}\int_{\mathbb{T}^3} \big(c_{0}|\nabla\sqrt{n}|^2+\kappa n|v-u|^2+\mu |\nabla u|^2+(\mu+\lambda)(\dive u)^2\big) dxd\tau\\
		&\quad \leq C+ C\underset{\var \to 0}{{\rm{lim~inf}}}~\int_{\mathbb{T}^3}( n_{\var}(\dive v_{\var})^2+n_{\var}|v_{\var}-u_{\var}|^2)dx-2\int_{0}^{t}\int_{\mathbb{T}^3}  \sqrt{n}v\cdot\nabla \sqrt{n}  dxd\tau\\
		&\quad\quad+\underset{\var \to 0}{{\rm{lim~inf}}}~\var\int_{\mathbb{T}^3}( n_{\var}|v_{\var}|^5+|\nabla \sqrt{n_{\var}}|^4+|\nabla \sqrt{n_{\var}}|^2 + |v_{\var}|^2|\nabla \sqrt{n_{\var}}|^2(1+|\nabla \sqrt{n_{\var}}|^2)+ n_{\var}^{-25})dx.
	\end{aligned}
\end{equation}  
The combination of $(\ref{factpp})$, $(\ref{4.28})$ and $(\ref{BDiiuniform})$ gives rise to $(\ref{BDuniformintime})$.

	\section{Vanishing artificial pressure}\label{section5}

	In this section, we are ready to show that the sequence of weak solutions to the IVP (\ref{twodelta}) converges to a expected weak solution to the original IVP (\ref{two}) as $\delta\rightarrow0$. Since the $L^{\infty}(0,T;L^{\gamma}(\mathbb{T}^3)$-bound of the density $\rho_{\delta}$ is not enough to get the convergence of the pressure, we need the higher integrability estimates, which can be proved similarly as in Lemma \ref{lemma34} and the arguments in \cite[Page 174]{lions2}. 
	\begin{lemma} \label{lemma51}
		Let $T>0$, and $(n,n v,\rho u,u)$ be a weak solution to the IVP $(\ref{twodelta})$ given by Proposition $\ref{weakdelta}$ for $t\in (0,T]$. Then, under the assumptions of Theorem \ref{theorem11},  we have
		\begin{equation}
			\int_{0}^{T}\int_{\mathbb{T}^3} (\rho^{\frac{5\gamma}{3}-1}+\delta \rho^{\gamma_{0}+\frac{2\gamma}{3}-1})dxdt \le  C,   \label{5.1}
		\end{equation}
		where $C>0$ is a constant independent of $\delta$.
	\end{lemma}

	\vspace{2ex}
	
	\textbf{Proof of Theorem \ref{theorem11} on the global existence of weak solutions.} Let $\delta\in(0,1)$, and $(n_{\delta},n_{\delta} v_{\delta},\rho_{\delta},\rho_{\delta} u_{\delta})$ be a global weak solution to the IVP (\ref{twoapp}) given by Proposition \ref{weakdelta}. For any $T>0$, with the help of the uniform estimates $(\ref{massdelta}$-$(\ref{BDuniformintime})$ and $(\ref{lemma51})$, there exists a limit $(n,m,\rho,\tilde{m})$ such that as $\delta\rightarrow 0$, it holds
	\begin{equation}\label{limitdelta1}
		\left\{
		\begin{aligned}
			&\delta\int_{0}^{T}\int_{\mathbb{T}^3}\rho_{\delta}^{\gamma_{0}}dxdt
			~\rightarrow 0,\\
			&\sqrt{n_{\delta}}\overset{\ast}{\rightharpoonup} \sqrt{n}\quad\quad\quad\text{in} \quad L^{\infty}(0,T; H^1(\mathbb{T}^3)),\\
			&~n_{\delta} v_{\delta}\overset{\ast}{\rightharpoonup}  m\quad\quad\quad~\text{in}\quad L^{\infty}(0,T;L^{\frac{3}{2}}(\mathbb{T}^3)),\\
			&~\rho_{\delta}\rightharpoonup \rho\quad\quad\quad\quad~~ \text{in}\quad L^{\frac{5\gamma}{3}-1}(0,T;L^{\frac{5\gamma}{3}-1}(\mathbb{T}^3)),\\
			&~u_{\delta} \rightharpoonup u\quad\quad\quad\quad~~\text{in}\quad L^2(0,T;H^1(\mathbb{T}^3)).
		\end{aligned}
		\right.
	\end{equation}
	By using the Aubin-Lions lemma as well as the uniform estimates (\ref{massdelta})-(\ref{melletdelta}), we can denote
	\begin{equation}\nonumber
		v:=
		\begin{cases}
			\frac{m}{n},\quad
			& \mbox{if $n>0,$ } \\
			0,\quad
			& \mbox{if $n=0,$}
		\end{cases}
	\end{equation}
	and have as $\delta\rightarrow0$ (up to a sequence) that
	\begin{equation}\label{limitdelta2}
		\left\{
		\begin{aligned}
			&\sqrt{n_{\delta}}\rightarrow \sqrt{n}\quad\quad\quad\quad~\text{in}\quad C([0,T];L^{p}(\mathbb{T}^3)),\quad p\in[1,6),\\
			&~n_{\delta}\rightarrow n\quad\quad\quad\quad\quad~~\text{in}\quad C([0,T];L^{p}(\mathbb{T}^3)),\quad p\in[1,3),\\
			&~n_{\delta}v_{\delta}\rightarrow nv ~\quad\quad\quad\quad\text{in}\quad L^2(0,T;L^{p}(\mathbb{T}^3)),\quad p\in [1,\frac{3}{2}),\\
			&\sqrt{n_{\delta}}v_{\delta}\rightarrow \sqrt{n}v\quad\quad\quad \text{in}\quad L^2(0,T;L^{2}(\mathbb{T}^3)),\\
			&~\rho_{\delta}\rightarrow \rho\quad\quad\quad\quad\quad\quad\text{in}\quad C([0,T];H^{-1}(\mathbb{T}^3)),\\
			&~\rho_{\delta}\rightarrow \rho\quad\quad\quad\quad\quad\quad\text{in}\quad C([0,T];L^{\gamma}_{weak}(\mathbb{T}^3)),\\
			&~\rho_{\delta} u_{\delta}\rightarrow \rho u\quad\quad\quad~\quad\text{in}\quad C([0,T];H^{-1}(\mathbb{T}^3)),\\
			&~\rho_{\delta} u_{\delta}\rightarrow \rho u\quad\quad\quad~\quad\text{in}\quad C([0,T];L^{\frac{2\gamma}{\gamma+1}}_{weak}(\mathbb{T}^3)),\\
		\end{aligned}
		\right.
	\end{equation}
	Since the proof of (\ref{limitdelta2}) is quite similar to the arguments in \cite{mellet1}, we omit the details for brevity.

	Then, we aim to prove
	\begin{equation}\label{strongrhoae1}
		\begin{aligned}
			&\rho_{\delta}\rightarrow \rho\quad\quad\text{a.e. in}\quad\mathbb{T}^3\times(0,T).
		\end{aligned}
	\end{equation}
	If (\ref{strongrhoae1}) holds, then it follows
	\begin{equation}\label{strongpressure1}
		\begin{aligned}
			&\rho_{\delta}^{\gamma}\rightarrow \rho^{\gamma} \quad\quad \text{in}\quad L^1(0,T;L^1(\mathbb{T}^3)).
		\end{aligned}
	\end{equation}
	Arguing similarly as the proof of Proposition \ref{weakdelta}, one can conclude from (\ref{dapp21})-(\ref{dapp22}), $(\ref{massdelta})$-$(\ref{melletdelta})$ and $(\ref{limitdelta1})$-(\ref{strongpressure1}) that the limit $(n,nv,\rho,\rho u)$ indeed is a global weak solution to the IVP (\ref{two})-(\ref{d}) in the sense of Definition \ref{defn11} satisfying $(\ref{BD})$.
	
	\vspace{2ex}
	
	To show the strong converge of $\rho_{\delta}$, we consider the following cut-off function (cf. \cite{feireisl1}):
	\begin{equation}\label{Tk}
		T_{k}(s):=
		\begin{cases}
			s,
			& \mbox{if $0\leq s \leq k,$ } \\
			\text{smooth and concave},
			& \mbox{if $k\leq s\leq 3k,$}\\
			2k,
			& \mbox{if $s\geq 3k$.}
		\end{cases}
	\end{equation}
	As in \cite{feireisl1,lions2}, we can prove the following property of effect viscous flux:
	\begin{lemma}\label{lemma52}
		For any $k>0$, we have
		\begin{equation}\label{effectvar}
			\begin{aligned}
				&\lim_{\delta\rightarrow 0}\int_{0}^{T}\int_{\mathbb{T}^{3}}\big{(}A\rho_{\delta}^{\gamma}+\delta \rho_{\delta}^{\gamma_{0}}-(2\mu+\lambda)\dive u_{\delta}\big{)}T_{k}(\rho_{\delta})dxdt\\
				&~~=\int_{0}^{T}\int_{\mathbb{T}^{3}}\big{(}A\overline{\rho^{\gamma}}+\delta \overline{\rho^{\gamma_{0}}}-(2\mu+\lambda)\dive u\big{)}\overline{T_{k}(\rho)}dxdt,
			\end{aligned}
		\end{equation}
		where $\overline{f}$ denotes the weak limit of $f_{\varepsilon}$ as $\delta\rightarrow0$.
	\end{lemma}

	\begin{lemma}\label{lemma53}
		For any $k>0$, there exists a constant $C$ independent of $\delta$ and $k$ such that
		\begin{equation}\label{osc}
			\begin{aligned}
				&\limsup_{\delta\rightarrow0} \int_{0}^{T}\int_{\mathbb{T}^3}|T_{k}(\rho_{\delta})-T_{k}(\rho)|^{\gamma+1}dxdt\leq C.
			\end{aligned}
		\end{equation}
	\end{lemma}
	\begin{proof}
		From the facts that $\rho^{\gamma}$ is convex and $T_{k}(\rho)$ is  concave, we have by $(\ref{effectvar})$ that
		\begin{equation}\label{5.9}
			\begin{aligned}
			&\limsup_{\delta\rightarrow 0} \int_{0}^{T}\int_{\mathbb{T}^3}|T_{k}(\rho_{\delta})-T_{k}(\rho)|^{\gamma+1}dxdt   \notag \\
			&\quad \le \lim_{\delta \to 0} \int_{0}^{T}\int_{\mathbb{T}^{3}} (\rho_{\delta}^{\gamma}-\rho^{\gamma})(T_{k}(\rho_{\delta})-T_{k}(\rho))dxdt+\lim_{\delta \to 0} \int_{0}^{T}\int_{\mathbb{T}^{3}} (\overline{\rho^{\gamma}}-\rho^{\gamma})(T_{k}(\rho)-\overline{T_{k}(\rho)})dxdt   \\
			&\quad=\lim_{\delta \to 0} \int_{0}^{T}\int_{\mathbb{T}^{3}} (\rho_{\delta}^{\gamma}T_{k}(\rho_{\delta})-\overline{\rho^{\gamma}}\overline{T_{k}(\rho)})dxdt, \\
			&\quad=\frac{(2\mu+\lambda)}{A}\limsup_{\delta \to 0} \int_{0}^{T}\int_{\mathbb{T}^{3}} [\dive u_{\delta}T_{k}(\rho_{\delta})-\dive u\overline{T_{k}(\rho)}]dxdt      \\
			&\quad\le C \sup \|\dive u_{\delta}\|_{L^{2}(0,T;L^{2})}\lim_{\delta \to 0}(\|T_{k}(\rho_{\delta})-T_{k}(\rho)\|_{L^{2}(0,T;L^{2})}+\|T_{k}(\rho)-\overline{T_{k}(\rho)}\|_{L^{2}(0,T;L^{2})})  .
			\end{aligned}
		\end{equation}
		The above estimate as well as \eqref{energyvar} and \eqref{5.1} yields \eqref{osc}.
	\end{proof}
	
	\vspace{2ex}

	\textbf{Proof of strong convergence of the density $\rho_{\delta}.$} Introduce a series of functions
	\begin{align}
		L_{k}(z)=
		\begin{cases}
			z\log z, &0 \le z \le k,  \\
			z\log k+z\int_{k}^{z} \frac{T_{k}(s)}{s^2} ds, &z \ge k. 
		\end{cases}   \label{92}
	\end{align}
According to \cite{lions2}, both $L_{k}(\rho)$ and $L_{k}(\rho_{\delta})$ satifies the renormalized properties$:$
	\begin{align}
		& L_{k}(\rho)_{t}+\dive (L_{k}(\rho)u)+T_{k}(\rho)\dive u=0, \ \ \ \ \ \ \ \  \text{in}  \ \mathcal{D}^{'}((0,T) \times \mathbb{T}^3), \label{93}  \\
		&L_{k}(\rho_{\delta})_{t}+\dive (L_{k}(\rho_{\delta})u_{\delta})+T_{k}(\rho_{\delta})\dive u_{\delta}=0, \ \ \text{in} \ \mathcal{D}^{'}((0,T) \times \mathbb{T}^3). \label{94}
	\end{align}
	To prove $\overline{L_{k}(\rho)}=L_{k}(\rho)$, we deduces from  $(\ref{93})$-$(\ref{94})$ that
	\begin{align}
		\int_{\mathbb{T}^3} (\overline{L_{k}(\rho)}-L_{k}(\rho))dx=\lim_{\delta \to 0} \int_{0}^{T}\int_{\mathbb{T}^{3}} (T_{k}(\rho)\dive u-T_{k}(\rho_{\delta})\dive u_{\delta})dxdt,    \label{97}
	\end{align}
	By $(\ref{effectvar})$ and $(\ref{97})$, it holds as $k\rightarrow\infty$ that
	\begin{equation} \label{98}
		\begin{aligned}
		&\int_{\mathbb{T}^3} (\overline{L_{k}(\rho)}-L_{k}(\rho))dx   \\
		&=\int_{0}^{T}\int_{\mathbb{T}^{3}} T_{k}(\rho)\dive u+\frac{1}{2\mu+\lambda}\lim_{\delta\rightarrow 0}\int_{0}^{T}\int_{\mathbb{T}^{3}}(A\rho_{\delta}^{\gamma}-(2\mu+\lambda)\dive u_{\delta})T_{k}(\rho_{\delta})dxdt  \\
		&\quad-\frac{A}{2\mu+\lambda}\lim_{\delta\rightarrow 0}\int_{0}^{T}\int_{\mathbb{T}^{3}}\rho_{\delta}^{\gamma}T_{k}(\rho_{\delta})dxdt  \\
		&\le \int_{0}^{T}\int_{ \{\rho \ge k \} } [T_{k}(\rho)-\overline{T_{k}(\rho)}]\dive udxdt+\int_{0}^{T}\int_{ \{\rho \le k \}} [T_{k}(\rho)-\overline{T_{k}(\rho)}]\dive udxdt    \\
		&\le  C\|\dive u\|_{L^2(\{\rho \ge k\})}+Ck^{-(\gamma-1)}  \to 0.    
		\end{aligned}
	\end{equation}
	Furthermore, we have as $k\rightarrow\infty$ that
	\begin{equation}\label{5.18}
		\begin{aligned}
		&\|L_k(\rho)-\rho\log \rho\|_{L^{1}(0,T;L^{1}(\mathbb{T}^{3}))} \\
		&\le C \int_{0}^{T}\int_{\{ \rho \ge k \} } |\rho\log \rho|dxdt \le C k^{-\gamma+\frac{3}{2}} \int_{0}^{T}\int_{ \{\rho \ge k \} } \rho^{\gamma}dxdt \to 0, 
		\end{aligned}
	\end{equation}
	and
	\begin{equation}\label{100}
		\begin{aligned}
		&\|L_k(\rho_{\delta})-\rho_{\delta}\log \rho_{\delta}\|_{L^{1}(0,T;L^{1}(\mathbb{T}^{3}))}
		 \le C  k^{-\gamma+\frac{3}{2}}\int_{0}^{T}\int_{ \{\rho \ge k \} } \rho_{\delta}^{\gamma} dxdt  \to 0. 
		\end{aligned}
	\end{equation}
	With the aid of $(\ref{5.18})$ and the lower semi-continuity of weak limits, it further holds
	\begin{align}
		\|\overline{L_k(\rho)}-\overline{\rho\log \rho}\|_{L^{1}(0,T;L^{1}(\mathbb{T}^{3}))} \le \lim_{\delta \to 0} \inf \|L_k(\rho_{\delta})-\rho_{\delta}\log \rho_{\delta}\|_{L^{1}((0,T) \times \mathbb{T}^{3})} \to 0, \ \ \text{uniformly} \ \text{in} \ \delta.  \label{101}
	\end{align}
	By $(\ref{98})$-$(\ref{101})$, we conclude that
	\begin{align}
		\int_{\mathbb{T}^3} (\overline{\rho\log \rho}-\rho\log\rho)dx \le 0,  \ \  a.e. \ \ t \in (0, T) \label{5.29}
	\end{align}
This combined with $\overline{\rho\log \rho} \ge \rho \log \rho$ leads to $(\ref{strongrhoae1})$. The proof of Theorem \ref{theorem11} is completed.

	\section{Large time behavior}\label{section6}
	
	In this section, we are ready to study the large time behavior of global weak solutions to the IVP $(\ref{two})$-$(\ref{d})$.

	\begin{lemma}
		Let the assumptions $(\ref{a1})$ be satisfied, and $(n,nv,\rho,\rho u)$ be the global weak solution to the IVP $(\ref{two})$-$(\ref{d})$ given by Theorem \ref{theorem11}. Then there holds for a.e. $0\leq s<t<\infty$~{\rm(}including $s=0${\rm)} that
		\begin{equation}\label{menergy}
			\begin{aligned}
				\widetilde{E}(t)+\int_{s}^{t}\int_{\mathbb{T}^3}\big( \kappa n|v-u|^2+\mu|\nabla u|^2+(\mu+\lambda)(\dive u)^2 \big) dxd\tau\leq \widetilde{E}(s),
			\end{aligned}
		\end{equation}
		where the modified energy $\widetilde{E}(t)$ is defined by
		\begin{equation}\label{widee}
			\begin{aligned}
				&\widetilde{E}(t):=\int_{\mathbb{T}^3}\Big{(} \frac{1}{2}n \big{|} v-m_{1}(t)\big{|}^2+n\log{n}-n+1+\frac{1}{2}\rho\big{|} u-m_{2}(t)\big{|}^2+\frac{A\rho^{\gamma}}{\gamma-1}\Big{)}dx\\
				&\quad\quad\quad+\frac{1}{2}C_{n_{0},\rho_{0}}|(m_{1}-m_{2})(t)|^2.
			\end{aligned}
		\end{equation}
		and $C_{n_{0},\rho_{0}}$ and $m_{i}(t), i=1,2,$ are given by
		\begin{equation}\label{m1m2}
			\begin{aligned}
				&C_{n_{0},\rho_{0}}:=\frac{\int_{\mathbb{T}^3} n_{0}dx\int_{\mathbb{T}^3}\rho_{0}dx}{\int_{\mathbb{T}^3}(n_{0}+\rho_{0})dx},\quad\quad m_{1}(t):=\frac{\int_{\mathbb{T}^3} nv dx}{\int_{\mathbb{T}^3} ndx} ,\quad\quad m_{2}(t):=\frac{\int_{\mathbb{T}^3} \rho u dx}{\int_{\mathbb{T}^3} \rho dx}.
			\end{aligned}
		\end{equation}
	\end{lemma}
	\begin{proof}
		
		One deduces from $(\ref{mass})_{2}$, (\ref{dapp21}), (\ref{ee1})-(\ref{ee2}) and (\ref{limitdelta2}) that
		\begin{align}
			& m_{1}(t)=\frac{1}{\int_{\mathbb{T}^3} n  dx}(\int_{\mathbb{T}^3}m_{0}dx+\int_{0}^{t}\int_{\mathbb{T}^3}\kappa (nu-nv)dxd\tau),\quad\quad t>0,\label{eee1}\\
			&m_{2}(t)=\frac{1}{\int_{\mathbb{T}^3} \rho  dx}(\int_{\mathbb{T}^3} \tilde{m}_{0}dx+\int_{0}^{t}\int_{\mathbb{T}^3}\kappa (nv-nu)dxd\tau),\quad\quad t>0,\label{eee2}
		\end{align}
		which together with $(\ref{mass})_{1}$ yields
		\begin{equation}\label{relativeuv}
			\begin{aligned}
				&\int_{\mathbb{T}^3}\big(\frac{1}{2}n|v-m_{1}|^2+\frac{1}{2}\rho|u-m_{2}|^2\big)dx\\
				&\quad=\int_{\mathbb{T}^3}\big(\frac{1}{2}n|v|^2+\frac{1}{2}\rho|u|^2\big)dx-\frac{1}{2}\big{(} m_{1}(t) \big{)}^2\int_{\mathbb{T}^3} ndx-\frac{1}{2} \big{(} m_{2}(t) \big{)}^2\int_{\mathbb{T}^3}\rho dx\\
				&\quad=\int_{\mathbb{T}^3}\big(\frac{1}{2}n|v|^2+\frac{1}{2}\rho|u|^2\big)dx-\frac{1}{2}\big{(} \frac{(\int_{\mathbb{T}^3} m_{0}dx)^2} {\int_{\mathbb{T}^3}n_{0}dx}+\frac{(\int_{\mathbb{T}^3} \tilde{m}_{0}dx)^2}{\int_{\mathbb{T}^3}\rho_{0}dx}\big) \\
				&\quad\quad-(\frac{\int_{\mathbb{T}^3} m_{0}}{\int_{\mathbb{T}^3} n_{0}dx}+\frac{\int_{\mathbb{T}^3} \tilde{m}_{0}}{\int_{\mathbb{T}^3} \rho_{0}dx})\int_{0}^{t}\int_{\mathbb{T}^3}\kappa (nu-nv)dxd\tau\\
				&\quad\quad-\frac{\int_{\mathbb{T}^3}(n_{0}+\rho_{0})dx}{\int_{\mathbb{T}^3}n_{0}dx\int_{\mathbb{T}^3}\rho_{0}dx} \Big{(} \int_{0}^{t}\int_{\mathbb{T}^3}\kappa (nu-nv)dxd\tau\Big{)}^2.
			\end{aligned}
		\end{equation}
		And we make use of (\ref{mass}) and (\ref{eee1})-(\ref{eee2}) again to obtain
		\begin{equation}\label{m1m2t}
			\begin{aligned}
				&\frac{1}{2}C_{n_{0},\rho_{0}}|(m_{1}-m_{2})(t)|^2\\
				&\quad=\frac{\int_{\mathbb{T}^3} n_{0}dx\int_{\mathbb{T}^3}\rho_{0}dx}{\int_{\mathbb{T}^3}(n_{0}+\rho_{0})dx} \Big{|} \frac{\int_{\mathbb{T}^3} m_{0}dx}{\int_{\mathbb{T}^3} n_{0}dx}-\frac{\int_{\mathbb{T}^3} \tilde{m}_{0}dx}{\int_{\mathbb{T}^3} \rho_{0}dx}+\frac{\int_{\mathbb{T}^3}(n_{0}+\rho_{0})dx}{\int_{\mathbb{T}^3}n_{0}dx\int_{\mathbb{T}^3}\rho_{0}dx} \int_{0}^{t}\int_{\mathbb{T}^3}\kappa (nu-nv)dxd\tau \Big{|}^2\\
				&\quad=\frac{1}{2}C_{n_{0},\rho_{0}}|(m_{1}-m_{2})(0)|^2+(\frac{\int_{\mathbb{T}^3} m_{0}}{\int_{\mathbb{T}^3} n_{0}dx}+\frac{\int_{\mathbb{T}^3} \tilde{m}_{0}}{\int_{\mathbb{T}^3} \rho_{0}dx})\int_{0}^{t}\int_{\mathbb{T}^3}\kappa (nu-nv)dxd\tau\\
				&\quad\quad+\frac{\int_{\mathbb{T}^3}(n_{0}+\rho_{0})dx}{\int_{\mathbb{T}^3}n_{0}dx\int_{\mathbb{T}^3}\rho_{0}dx} \Big{(} \int_{0}^{t}\int_{\mathbb{T}^3}\kappa (nu-nv)dxd\tau\Big{)}^2.
			\end{aligned}
		\end{equation}
		By (\ref{energyin}) and (\ref{relativeuv})-(\ref{m1m2t}), (\ref{menergy}) holds for $s=0$.

		In addition, owing to (\ref{menergy}) for $s=0$, we deduce for any nonnegative text function $\psi=\psi(t)\in \mathcal{D}((0,T))$ that
		\begin{equation}\label{epp2}
			\begin{aligned}
				&-\int_{0}^{\infty}\psi_{t} \widetilde{E}(t)dt+\int_{0}^{\infty}\psi\int_{\mathbb{T}^3}\big( \kappa n|v-u|^2+\mu|\nabla u|^2+(\mu+\lambda)(\dive u)^2 \big) dxdt\leq 0,
			\end{aligned}
		\end{equation}
		Let $\psi_{\var_{*}}=\psi_{\var_{*}}(t)\in \mathcal{D}(0,\infty)$ for $\var_{*}\in(0,1)$ be the Friedrichs mollifier. Then letting $\psi(t)=\psi_{\var_{*}}(t-\cdot)$ in (\ref{epp2}), we deduce for any $0<s\leq \tau\leq t<\infty$ that
		\begin{equation}\label{epp3}
			\begin{aligned}
				&\frac{d}{dt} \widetilde{E}\ast \psi_{\var_{*}}(\tau)+\Big{(}\int_{\mathbb{T}^3}\big( \kappa n|v-u|^2+\mu|\nabla u|^2+(\mu+\lambda)(\dive u)^2 \big) dx\Big{)} \ast \psi_{\var_{*}}(\tau)\leq 0.
			\end{aligned}
		\end{equation}
		Integrating (\ref{epp3}) over $[s,t]$ and taking the limit as $\var_{*}\rightarrow 0$, we prove (\ref{menergy}) for a.e. $0<s<t<\infty$.

	\end{proof}

	\textbf{Proof of Theorem \ref{theorem12} on the large time behavior of weak solutions.} Let the assumptions of Theorem \ref{theorem11} hold, and $(n,nv,\rho,u)$ be a global weak solution to the IVP \eqref{two}-\eqref{d} given by Theorem \ref{theorem11}. For any $s\geq0$ and $t\in(0,1)$, denote the sequence
	\begin{equation}\nonumber
		\begin{aligned}
			&(n_{s},n_{s}v_{s},\rho_{s}, u_{s})(x,t):=(n,nv,\rho,u)(x,t+s).
		\end{aligned}
	\end{equation}
	By (\ref{energyin}), (\ref{mass}), $(\ref{BD})_{1}$ and (\ref{menergy}), we get the uniform bounds
	\begin{equation}\label{Ein1}
	\left\{
		\begin{aligned}
			&\sup_{s\geq0}\big(\|\sqrt{n_{s}}\|_{L^{\infty}(0,1;H^1)}+\|\nabla\sqrt{n_{s}}\|_{L^{2}(0,1;L^2)}+\|\nabla n_{s}\|_{L^2(0,1;L^{\frac{3}{2}})}\big)<\infty,\\
			&\sup_{s\geq0} \big(\|\sqrt{n_{s}}v_{s}\|_{L^{\infty}(0,1;L^2)}+\|\rho_{s}\|_{L^{\infty}(0,1;L^{\gamma})}+\|\sqrt{\rho_{s}}u_{s}\|_{L^{\infty}(0,1;L^2)}\big)<\infty,
		\end{aligned}
		\right.
	\end{equation}
	and the strong convergences
	\begin{equation}\label{Dstrongt1}
		\left\{
		\begin{aligned}
			&\lim_{s\rightarrow\infty}(\|\nabla n_{s}\|_{L^2(0,1;L^{\frac{3}{2}})}+\|n_{s}\|_{L^{\infty}(0,1;L^1)}+\|\sqrt{n_{s}}(v_{s}-u_{s})\|_{L^2(0,1;L^2)}+ \|\nabla u_{s}\|_{L^2(0,1;L^2)})=0,\\
			&\lim_{s\rightarrow\infty}\|\frac{\int_{\mathbb{T}^3} n_{s}v_{s} dx}{\int_{\mathbb{T}^3}n_{s}dx}-\frac{\int_{\mathbb{T}^3}\rho_{s} u_{s}dx}{\int_{\mathbb{T}^3}\rho_{s} dx}\|_{L^2(0,1)}=0.
		\end{aligned}
		\right.
	\end{equation}
	Owing to (\ref{Ein1}), $(\ref{mass})$, $(\ref{Dstrongt1})_{1}$ and the Sobolev inequality, it holds as $s\rightarrow\infty$ that
	\begin{equation}\label{Dstrongt}
		\left\{
		\begin{aligned}
			&\|n_{s}-n_{c}\|_{L^2(0,T;L^{3})}\rightarrow0,\quad\quad n_{c}:=\int_{\mathbb{T}^3}n_{0}dx,\\
			&\|u_{s}-\frac{\int_{\mathbb{T}^3} \rho_{s} u_{s}dx}{\int_{\mathbb{T}^3} \rho_{s} dx}\|_{L^2(0,1;L^{6})}\\
			&\quad\leq (1+\frac{\sup_{t\in(0,1)}\|\rho_{s}(t)\|_{L^{\frac{6}{5}}}}{\|\rho_{0}\|_{L^1}}) \|u_{s}-\int_{\mathbb{T}^3}u_{s}dx\|_{L^2(0,1;L^{6})}\rightarrow 0,\\
			&\|\sqrt{n_{s}}(v_{s}-\frac{\int \rho_{s} u_{s}dx}{\int \rho_{s} dx})\|_{L^2(0,1;L^2)}\\
			&\quad\leq  \|\sqrt{n_{s}}(v_{s}-u_{s})\|_{L^2(0,1;L^2)}+\sup_{t\in(0,1)}\|\sqrt{n_{s}}(t)\|_{L^3}\|u-\frac{\int_{\mathbb{T}^3} \rho_{s} u_{s}dx}{\int_{\mathbb{T}^3} \rho_{s} dx}\|_{L^2(0,1;L^{6})}\rightarrow0.
		\end{aligned}
		\right.
	\end{equation}
	Since $(n_{s})_{t}=-\dive (\sqrt{n_{s}}\sqrt{n_{s}} u_{s})$ is uniformly bounded in $L^{\infty}(0,1; W^{-1,\frac{3}{2}}(\mathbb{T}^3))$, by virtue of \eqref{Ein1}, $(\ref{Dstrongt})_{1}$ and the Aubin-Lions lemma, there is a subsequence of $n_{s}$ (still denoted by the same index) such that as $s\rightarrow\infty$, we have
	\begin{equation}\label{strongns}
		\begin{aligned}
			&n_{s}\rightarrow n_{c}\quad \quad\text{in}\quad C([0,T]; L^{q}(\mathbb{T}^3)),\quad q\in[1,3).
		\end{aligned}
	\end{equation}
	Then, by similar arguments as used in $(\ref{lemma51})$, we are able to show
	\begin{equation}\label{rho53}
		\begin{aligned}
			&\int_{0}^{1}\int_{\mathbb{T}^3}\rho_{s}^{\frac{5\gamma}{3}-1}dxdt\leq C,
		\end{aligned}
	\end{equation}
	for a constant $C>0$ independent of $s$, so there exist two functions $\overline{\rho}\in L^{\frac{5\gamma}{3}-1}(0,1;L^{\frac{5\gamma}{3}-1}(\mathbb{T}^3))$ and $\overline{\rho^{\gamma}} \in L^{\frac{5}{3}-\frac{1}{\gamma}}(0,1;L^{\frac{5}{3}-\frac{1}{\gamma}}(\mathbb{T}^3))$ such that up to a sequence if necessary, it holds as $s\rightarrow\infty$ that
	\begin{equation}\nonumber
		\left\{
		\begin{aligned}
			&\rho_{s}\rightharpoonup \overline{\rho}\quad\quad \quad\text{in}\quad L^{\frac{5\gamma}{3}-1}(0,1;L^{\frac{5\gamma}{3}-1}(\mathbb{T}^3)),\\
			&\rho_{s}^{\gamma}\rightharpoonup \overline{\rho^{\gamma}}\quad\quad~\text{in}\quad L^{\frac{5}{3}-\frac{1}{\gamma}}(0,1;L^{\frac{5}{3}-\frac{1}{\gamma}}(\mathbb{T}^3)).
		\end{aligned}
		\right.
	\end{equation}
	Set
	$$
	G(P)=P^{\alpha},\quad~~  P\in \mathbb{R}_{+},\quad\quad 0<\alpha<\frac{5}{3}-\frac{1}{\gamma},
	$$
	so that by (\ref{rho53}), as $s\rightarrow\infty$, there exists two limits $\overline{G(\rho)}\in L^{p_{1}}(0,1;L^{p_{1}}(\mathbb{T}^3))$ and $\overline{G(\rho^{\gamma})\rho^{\gamma}}\in L^{p_{2}}(0,1;L^{p_{2}}(\mathbb{T}^3))$ with $\frac{1}{p_{1}}+\frac{1}{p_{2}}<1$ satisfying
	\begin{equation}\nonumber
		\left\{
		\begin{aligned}
			&G(\rho_{s}^{\gamma})\rightharpoonup\overline{G(\rho)}\quad\quad\quad\quad\text{in}\quad L^{p_{1}}(0,1;L^{p_{1}}(\mathbb{T}^3)),\\
			&G(\rho_{s}^{\gamma})\rho_{s}^{\gamma}\rightharpoonup\overline{G(\rho)\rho^{\gamma}}\quad\quad\text{in}\quad L^{p_{2}}(0,1;L^{p_{2}}(\mathbb{T}^3)).
		\end{aligned}
		\right.
	\end{equation}
	
	As in \cite{feireisl4}, we apply the dive-curl lemma to get the strong convergence of $\rho_{s}$. Consider the two vector functions
	$$
	\big( G(\rho_{s}^{\gamma}), 0, 0,0\big),\quad\quad \big( \rho_{s}^{\gamma}, 0,0,0).
	$$
	By the arguments of renormalized solutions of (\ref{e3}), it holds
	\begin{equation}\nonumber
		\begin{aligned}
			&G(\rho_{s}^{\gamma})_{t}=-\dive( G(\rho_{s}^{\gamma})u_{s})-(\gamma\alpha-1)G(\rho_{s}^{\gamma})\dive u_{s}\quad\quad\text{in}\quad\mathcal{D}'(\mathbb{T}^3\times(0,1)),
		\end{aligned}
	\end{equation}
	which implies that 
	\begin{equation}\nonumber
		\begin{aligned}
			\dive_{t,x}\big( G(\rho_{s}^{\gamma}), 0, 0,0\big)=G(\rho_{s}^{\gamma})_{t}
		\end{aligned}
	\end{equation}
	is strongly compact in $W^{-1,q_{1}}(\mathbb{T}^3\times (0,1)$ for some $q_{1}>1$. It is easy to check that
	\begin{equation}\nonumber
		\begin{aligned}
			{\rm{curl}}_{t,x}\big( \rho_{s}^{\gamma}, 0,0,0\big) =
			{\tiny\left(
					\begin{matrix}
						0& -\partial_{x_{1}}\rho_{s}^{\gamma}  &-\partial_{x_{2}}\rho_{s}^{\gamma} & -\partial_{x_{3}}\rho_{s}^{\gamma}\\
						\partial_{x_{1}}\rho_{s}^{\gamma} &0  &0 & 0\\
						\partial_{x_{2}}\rho_{s}^{\gamma} &0  &0 &0\\
						\partial_{x_{3}}\rho_{s}^{\gamma} &0 &0 &0\\
					\end{matrix}
					\right)}
		\end{aligned}
	\end{equation}
	is strong compact in $W^{-1,q_{2}}(\mathbb{T}^3\times (0,1)$ for some $q_{2}>1$. Therefore, employing the $L^{p}$ type dive-curl lemma (cf. \cite{zou1}), we have
	$$
	\overline{G(\rho^{\gamma})\rho^{\gamma}}=\overline{G(\rho^{\gamma})}\overline{\rho^{\gamma}}.
	$$
	Since $G(P)$ is strictly monotone with respect to the variable $P$, we deduce
	\begin{equation}\label{Glimit}
		\begin{aligned}
			\overline{G(\rho^{\gamma})}=G(\overline{\rho^{\gamma}}),
		\end{aligned}
	\end{equation}
	By (\ref{Glimit}) and the convexity of $L^{\frac{1}{\alpha}}$, it holds
	\begin{equation}\nonumber
		\begin{aligned}
			\rho_{s}^{\gamma}\rightarrow \overline{\rho^{\gamma}}\quad \quad\text{in}\quad L^1(0,1;L^1(\mathbb{T}^3)),
		\end{aligned}
	\end{equation}
	which together (\ref{mass}) and (\ref{rho53}) gives rise to
	\begin{equation}\label{strongrhos}
		\begin{aligned}
			&\rho_{s}\rightarrow \rho_{c}:=\int_{\mathbb{T}^3}\rho_{0}dx\quad \quad\text{in}\quad L^{p}(0,1;L^{p}(\mathbb{T}^3)),\quad\quad 1\leq p<\frac{5}{3}\gamma-1.
		\end{aligned}
	\end{equation}
	Since $(\rho_{s})_{t}$ is uniform bounded in $L^{\infty}(0,1;L^{\frac{2\gamma}{\gamma+1}}(\mathbb{T}^3))$, we further have
	\begin{equation}\label{strongrhos1}
		\begin{aligned}
			&\rho_{s}\rightarrow \rho_{c}\quad \quad\text{in}\quad C([0,1];L^{\gamma}_{weak}(\mathbb{T}^3)).
		\end{aligned}
	\end{equation}

	Finally,  due to (\ref{menergy}), the energy $\widetilde{E}(t)$ defined by (\ref{widee}) is non-increasing and bounded from below, and therefore $\widetilde{E}(t)$ converge to a constant $\widetilde{E}_{\infty}\geq0$ as $t\rightarrow\infty$:
	$$
	\widetilde{E}_{\infty}:=\limsup_{t\rightarrow\infty}\widetilde{E}(t).
	$$
	By (\ref{strongrhos})-(\ref{strongrhos1}) and the Fatou lemma, there is a sequence $s_{k}\rightarrow \infty$ such that
	\begin{equation}\nonumber
		\begin{aligned}
			&\widetilde{E}_{\infty}=\lim_{s_{k}\rightarrow\infty}\int_{s_{k}}^{s_{k}+1}\widetilde{E}(\tau)d\tau\\
			&\quad~~=\lim_{s_{k}\rightarrow\infty}\int_{0}^{1}\int_{\mathbb{T}^3}\Big(\frac{1}{2}n_{s}\big|v_{s}-\frac{\int_{\mathbb{T}^3} n_{s}v_{s}dx}{\int_{\mathbb{T}^3}n_{s}dx} \big|^2+n_{s}\log{n_{s}}-n_{s}+1\\
			&\quad\quad~~+\frac{1}{2}\rho_{s}\big |u_{s}-\frac{\int_{\mathbb{T}^3}\rho_{s}u_{s}dx}{\int_{\mathbb{T}^3}\rho_{s}dx} \big|^2+\frac{A\rho_{s}^{\gamma}}{\gamma-1}\Big)dxd\tau\\
			&\quad\quad ~~+\frac{\int_{\mathbb{T}^3} n_{0}dx\int_{\mathbb{T}^3}\rho_{0}dx}{2\int_{\mathbb{T}^3}(n_{0}+\rho_{0})dx}\lim_{s_{k}\rightarrow\infty}\int_{0}^{1}\big{|}\frac{\int_{\mathbb{T}^3} n_{s}v_{s} dx}{\int_{\mathbb{T}^3} n_{s}dx}-\frac{\int_{\mathbb{T}^3} \rho_{s} u_{s} dx}{\int_{\mathbb{T}^3} \rho_{s} dx}\big{|}^2d\tau\\
			&\quad~~=\int_{\mathbb{T}^3}\big( n_{c}\log{n_{c}}-n_{c}+1+\frac{A\rho_{c}^{\gamma}}{\gamma-1} \big) dx\\
			&\quad~~\leq \liminf_{t\rightarrow \infty}\int_{\mathbb{T}^3}\big(n\log{n}-n+1+\frac{A\rho^{\gamma}}{\gamma-1} \big)dx\leq \limsup_{t\rightarrow\infty}\widetilde{E}(t)=\widetilde{E}_{\infty},
		\end{aligned}
	\end{equation}
so we have 
\begin{equation}
	\lim_{t\rightarrow \infty}\int_{\mathbb{T}^3}\big(n\log{n}-n+1+\frac{A\rho^{\gamma}}{\gamma-1} \big)dx=\int_{\mathbb{T}^3}\big( n_{c}\log{n_{c}}-n_{c}+1+\frac{A\rho_{c}^{\gamma}}{\gamma-1} \big) dx=\widetilde{E}_{\infty}. \label{wehave}
\end{equation}
	which together with (\ref{strongns}) gives rise to
	\begin{equation}\label{strongrhof}
		\begin{aligned}
			&\lim_{t\rightarrow\infty}\int_{\mathbb{T}^3}|\rho-\rho_{c}|^{\gamma}dx=0.
		\end{aligned}
	\end{equation}
	With the aid of  (\ref{wehave})-(\ref{strongrhof}), it further holds
	\begin{equation}\label{stronguv}
		\left\{
		\begin{aligned}
			&\lim_{t\rightarrow\infty}\int_{\mathbb{T}^3}(n|v-\frac{\int_{\mathbb{T}^3} nv dx}{\int_{\mathbb{T}^3}n dx}|^2+\rho|u-\frac{\int_{\mathbb{T}^3} \rho udx}{\int_{\mathbb{T}^3} \rho dx}|^2)dx=0,\\
			&\lim_{t\rightarrow\infty}\big{|}\frac{\int_{\mathbb{T}^3} nv dx}{\int_{\mathbb{T}^3} ndx}-\frac{\int_{\mathbb{T}^3} \rho u dx}{\int_{\mathbb{T}^3} \rho dx}\big{|}=0.
		\end{aligned}
		\right.
	\end{equation}
	Making use of the conservation laws (\ref{mass}), we substitute the equality
	\begin{equation}\nonumber
		\begin{aligned}
			&\frac{\int_{\mathbb{T}^3} \rho udx}{\int_{\mathbb{T}^3} \rho dx}=\frac{1}{\int_{\mathbb{T}^3}\rho_{0}dx}( \int_{\mathbb{T}^3}(m_{0}+\tilde{m_{0}})dx-\int_{\mathbb{T}^3}n_{0}dx \frac{\int_{\mathbb{T}^3} nv dx}{\int_{\mathbb{T}^3}n dx})
		\end{aligned}
	\end{equation}
	into $(\ref{stronguv})_{2}$ to derive
	$$
	\lim_{t\rightarrow\infty}\big(\|\frac{\int_{\mathbb{T}^3} nv dx}{\int_{\mathbb{T}^3}ndx}-u_{c}\|_{L^2(0,1)}+\|\frac{\int_{\mathbb{T}^3}\rho u dx}{\int_{\mathbb{T}^3}\rho dx}-u_{c}\|_{L^2(0,1)} \big)=0,\quad u_{c}:=\frac{\int_{\mathbb{T}^3}(m_{0}+\tilde{m}_{0})dx}{\int_{\mathbb{T}^3} (n_{0}+\rho_{0})dx}.
	$$
	This together with $(\ref{stronguv})$ yields
	\begin{equation}\label{Dstrongt12}
		\begin{aligned}
			&\lim_{t\rightarrow\infty}\int_{\mathbb{T}^3}(n|v-u_{c}|^2+\rho|u-u_{c}|^2)dx=0.
		\end{aligned}
	\end{equation}
	The combination of (\ref{strongns}), (\ref{strongrhof}) and (\ref{Dstrongt12}) gives rise to (\ref{decay1}). The proof of Theorem \ref{theorem12} is completed.
	\vspace{2ex}

	\noindent
	\textbf{Acknowledgments.} 
	The authors would like to thank Professor Hai-Liang Li for his helpful discussions and comments.
	The research of the paper is supported by National Natural Science Foundation of China (No.11931010, 11871047 and 11671384) and by the key research project of Academy for Multidisciplinary Studies, Capital Normal University, and by the Capacity Building for Sci-Tech Innovation-Fundamental Scientific Research Funds (No.007/20530290068).

\end{document}